\documentclass{article}
\usepackage[utf8]{inputenc}
\usepackage[T1]{fontenc}
\usepackage{amsmath,amssymb,amsthm,bbm,float,graphicx,geometry,mathtools,parskip,setspace,subcaption, lmodern}
\usepackage{mathrsfs}
\usepackage{enumerate}
\usepackage{nicefrac}
\usepackage{todonotes}
\usepackage{comment}
\usepackage[colorlinks, linkcolor = blue!80!black, citecolor = blue!80!black, breaklinks, pdfauthor={Benedikt Jahnel, Lukas Luechtrath, Peter Gracar, Anh Duc Vu}]{hyperref}
\usepackage{orcidlink}
\usepackage{titling}
\usepackage{url}
\usepackage{tikz}
\usetikzlibrary{arrows.meta}
\usetikzlibrary{calc}
\usepackage{dsfont}

\usepackage{authblk}

\usepackage{tikz}
\usetikzlibrary{backgrounds}
\usetikzlibrary{patterns}
\usetikzlibrary{positioning, shapes.geometric}

\usepackage{caption}
\usepackage{graphicx}
\usepackage{subcaption}
\usepackage{graphicx}
\graphicspath{{Images/}}%put any images in this file
\allowdisplaybreaks%breaks long equations onto multiple pages is needed
\newcommand{\N}{\mathbb{N}}
\DeclareMathOperator{\argmax}{arg \ max}
\usepackage{nameref}

\usepackage{cleveref}

\newtheorem{theorem}{Theorem}[section]
\crefname{theorem}{theorem}{theorems}
\newtheorem{lemma}[theorem]{Lemma}
\crefname{lemma}{lemma}{lemmas}

\crefname{corollary}{corollary}{corollaries}
\newtheorem{prop}[theorem]{Proposition}
\crefname{prop}{proposition}{propositions}

\crefname{conjecture}{conjecture}{conjectures}

\crefname{assumption}{assumption}{assumptions}

\newtheorem{definition}[theorem]{Definition}
\crefname{definition}{definition}{definitions}
\newtheorem{remark}[theorem]{Remark}
\crefname{remark}{remark}{remarks}

\makeatletter
\let\orgdescriptionlabel\descriptionlabel
\renewcommand*{\descriptionlabel}[1]{%
  \let\orglabel\label
  \let\label\@gobble
  \phantomsection
  \edef\@currentlabel{#1}%
  \let\label\orglabel
  \orgdescriptionlabel{#1}%
}

\renewcommand{\P}{\mathbb{P}}
\newcommand{\x}{\mathbf{x}}

\newcommand{\1}{\mathbbm{1}}
\newcommand{\G}{\mathscr{G}}
\newcommand{\C}{\mathscr{C}}

\newcommand{\E}{\mathbb{E}}

\newcommand{\R}{\mathbb{R}}
\renewcommand{\d}{\mathrm{d}}

\renewcommand{\d}{\mathrm{d}}

\newcommand{\Ecal}{\mathcal{E}}

\newcommand{\Ccal}{\mathcal{C}}
\newcommand{\eps}{\varepsilon}
\newcommand{\cP}{\mathcal{P}}

\newcommand{\Hcal}{\mathcal{H}}

\newcommand{\perc}{\mathrm{perc}}

\renewcommand{\Cap}{\mathrm{Cap}}
\newcommand{\dx}{\mathrm{d}}
\newcommand{\I}{\mathds{1}}

\pretitle{\centering\LARGE\scshape}
 \posttitle{\vskip 0.75cm}

 \predate{\vskip 0.5 cm \centering\large}
 \postdate{\par}

\usepackage[style = numeric, sorting=nyt, url = false, abbreviate=false, maxbibnames=9, sortcites=true, doi = true, backend = biber, giveninits = true, isbn=false]{biblatex}
\renewbibmacro{in:}{\ifentrytype{article}{}{\printtext{\bibstring{in}\intitlepunct}}}
\DeclareFieldFormat{journaltitle}{\mkbibemph{#1\isdot}}
\title{Detection, coverage and percolation in dynamic Boolean models with random radii based on $\alpha$-stable processes}
\author[1]{Peter Gracar\thanks{Email: p.gracar@leeds.ac.uk; \url{https://orcid.org/0000-0001-8340-8340}}}
\author[2,3]{Benedikt Jahnel\thanks{Email: benedikt.jahnel@tu-braunschweig.de; \url{https://orcid.org/0000-0002-4212-0065}}}
\author[3]{Lukas L\"uchtrath\thanks{Email: lukas.luechtrath@wias-berlin.de; \url{https://orcid.org/0000-0003-4969-806X}}}
\author[3]{Anh Duc Vu\thanks{Email: anhduc.vu@wias-berlin.de; \url{https://orcid.org/0009-0005-6913-4992}}}

\affil[1]{\small School of Mathematics, University of Leeds, Leeds LS2 9JT, UK}
\affil[2]{Technische Universit\"at Braunschweig, Braunschweig, Germany}
\affil[3]{Weierstrass Institute for Applied Analysis and Stochastics, Berlin, Germany}

\begin{document}
\date{\today}
\maketitle

\begin{abstract}\noindent
We consider a dynamic network in continuum time and space in which nodes, with initial locations given by a Poisson point process, move according to i.i.d.\ isotropic $\alpha$-stable processes. Each node is additionally equipped with an i.i.d.\ detection radius. Inspired by corresponding results by Peres et.~al.\ on mobile networks based on Brownian sausages with fixed width, we investigate the tail behaviour of three stopping times: The detection time of the first discovery of a designated node, the first coverage of an entire set, and the first discovery of a node by the infinite connected component of the system. Broadly speaking, we discover that the stability index as well as the random radii manifest themselves only in constants in the otherwise exponential decay rates. The proofs rest on heat-kernel bounds for the underlying Lévy processes and a detailed multiscale analysis allowing us to control the space-time correlations of the system. 

\bigskip
\footnotesize{
\noindent\textbf{AMS-MSC 2020}: Primary: 60D05; Secondary: 60K35

\medskip 

}
\end{abstract}

\pagebreak

\section{Introduction}\label{sec:intro}
Mobile geometric graphs are fundamental models for dynamic spatial networks, in which nodes are distributed in space and form edges with nearby nodes while moving according to some stochastic dynamics~\cite{van1997dynamic,diaz2009large,steif2009survey}. Understanding how mobility and connectivity interact is crucial in many applications, including wireless sensor networks, opinion dynamics, and epidemiology, see for example~\cite{kesidis2003surveillance,gupta1998critical,grossglauser2002mobility,liu2005mobility,ferri2023three,grimmett2022brownian,baldasso2022local}.

A rigorous mathematical framework is introduced in~\cite{Peres2011}, where nodes are initially placed according to a homogeneous Poisson point process in $\mathbb{R}^d$ and move independently according to {\em Brownian motions}. Nodes are connected if their Euclidean distance is less than a fixed radius $r>0$. They studied three fundamental quantities: \emph{detection time}, the first moment when a target (fixed or mobile) is within range of some node; \emph{coverage time}, the time until a finite region has been fully explored by the nodes; and \emph{percolation time}, the time until a typical node becomes part of the infinite connected component. Using stochastic geometry, multiscale coupling, and concentration techniques, they derive dimension-dependent asymptotics for the upper-tail probabilities of these times for a static or dynamic typical network participant. 

In parallel, networks based on \emph{Lévy jump processes}, or Lévy flights, have attracted attention due to their heavy-tailed displacement distributions and scale-invariance properties, which provide good fits to a large variety of behaviours observed in nature and engineering~\cite{gonzalez2008understanding,rhee2011levy,sims2008scaling,gan2021levy}.

In this work, we combine these strands by considering a dynamic geometric network, in which nodes follow Lévy jump processes rather than Brownian motions and radii are i.i.d.\ {\em random variables} instead of being globally fixed. This setting captures both heavy-tailed mobility and heterogeneous interaction ranges, which are common in real-world networks such as wireless communication systems and animal movement networks. By extending multiscale coupling and stochastic-geometric techniques to this richer model, we characterise the asymptotic behaviour of detection, coverage, and percolation times, generalising the findings in~\cite{Peres2011} to a more realistic non-Gaussian and heterogeneous setting.
%, and provide insights into how heavy-tailed dynamics and spatial heterogeneity jointly influence the temporal evolution of connectivity.

The manuscript is organised as follows. In \Cref{sec:setting} we first present the setting and some basic properties of the underlying processes that we will frequently make use of. Then, we exhibit our main findings on the detection, coverage, and percolation times including some case studies and provide an outlook for possible future research directions. The proofs are presented in \Cref{sec:proofs}.

%%%%%%%%%%%%%%%%%%%%%%%%%%%%%%%%%%%%%%%%%%%%%%%%%%%%%%%%%%%%%%%%%%%%%%%%
\section{Setting and main results}\label{sec:setting}
Fix the dimension $d\ge 1$, an intensity $\lambda>0$ and a stability index $\alpha\in(0,2]$. At time $t=0$, the node locations $\Psi_0=\{\Psi^i\}_{i\in \N}$ form a homogeneous {\em Poisson point process} of intensity~$\lambda$ on $\R^{d}$. We equip each node $\Psi^i$, \(i\in\N\), with an i.i.d.\ copy $R_i$ of a {\em communication radius} \(R\) that satisfies
      \begin{equation}\label{eq:moments}
         0 < \E R^d<\infty.
      \end{equation} 

Next, let $X=(X_t)_{t\geq0}$ be the \emph{standard $d$-dimensional isotropic $\alpha$-stable process} started in the origin $o\in\R^d$. That is, the {\em Lévy process} with characteristic function given by the  $\alpha$-stable law
\[
	\mathbb{E}\exp(\mathrm{i}\langle\xi,X_t\rangle)
          =\exp(-t\lvert\xi\rvert^{\alpha}),\qquad
         \forall \xi\in\mathbb{R}^{d},
\]
with {\em stability index} $\alpha\in(0,2]$. 
Recall that, as a Lévy process, it has stationary, independent increments and càdlàg paths and obeys the self-similarity (scaling) relation
\begin{equation}\label{eq:self-similarity-alpha-stable}
   (X_{ct})_{t\ge 0}
   \stackrel{d}{=}
   (c^{1/\alpha} X_t)_{t\ge 0}, \qquad \forall c>0. 
\end{equation}
Consequently, for \(\alpha<2\), each marginal $X_t$ has polynomial tails,
$\mathbb{P}(|X_t|>x)\sim C_\alpha t x^{-\alpha}$, as $x\to\infty$ for some $C_\alpha>0$,
so in particular $\operatorname{Var}|X_t|=\infty$ and, in case $\alpha \le  1$, even $\E|X_t|=\infty$.
The generator of the isotropic $\alpha$-stable process is the fractional Laplacian $-(-\Delta)^{\alpha/2}$ so that the $\alpha=2$ case corresponds to a standard Brownian motion at twice the speed.

We now equip the nodes of the Poisson point process \(\Psi_0\) with i.i.d.\ copies \(\{Y^i\}_{i\in\N}\) of \(X\) and define \(X^i:=Y^i+\Psi^i\). That is, each node of \(\Psi_0\) moves independently according to a standard \(\alpha\)-stable process. We denote the collection of all locations at time \(t\) by \(\Psi_t=\{X^i_t\}_{i\in \N}\). Thus,  \(\Psi_t\) has the law of a homogeneous Poisson point process with $\lambda>0$. Note that the associated communication radii do not change over time. The main object of investigation is the \emph{dynamic Boolean model} $\G=(\G_t)_{t\ge 0}$, given by
\[
    \G_t := \bigcup_{i\geq1} B_{R_i}(X_t^i)\subset\R^d,\qquad t\geq0,
\]
where $B_r(x)$ denotes the ball with radius $r\ge 0$ centred at $x\in \R^d$,
see \Cref{fig:process} for an illustration. 
The moment bound on the radius guarantees that this model is nontrivial, i.e.\ $\G_t\neq\R^d$ at each point in time \(t\geq 0\), see~\cite{Meester1996a, van1997dynamic}.
\begin{figure}[htbp]
        \includegraphics[width=\linewidth]{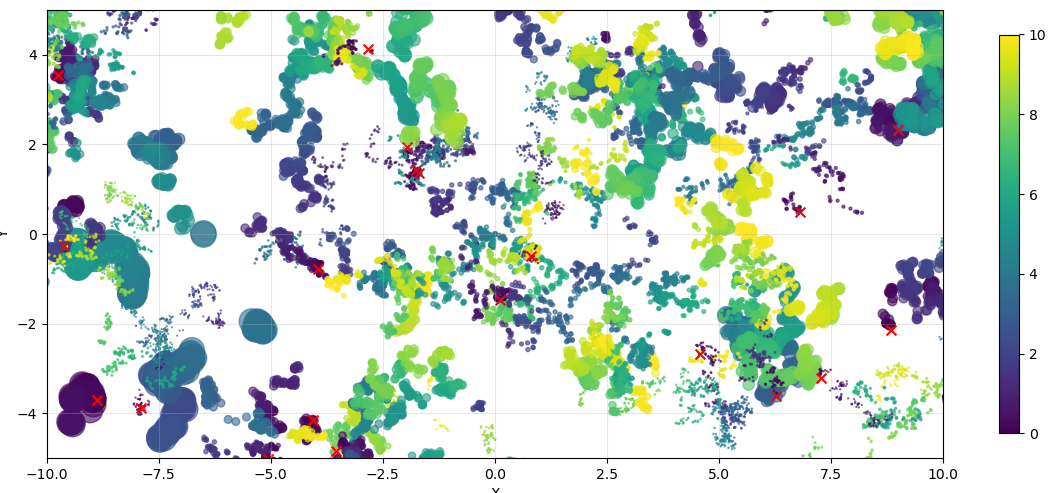}
    \caption{
    A realisation of the dynamic Boolean model $\G$ based on the standard isotropic $\alpha$-stable process with $\alpha= 1.5$. Colour indicates time. As a Boolean model of a stationary Poisson point process, $\G_t$ has the same distribution as $\G_0$ for any $t\geq0$. Plotted is $(\G_t)_{t\in[0,10]}$. The starting configuration is depicted by red crosses.}
    \label{fig:process}
\end{figure}

\subsection{Detection times}
Our first main result is concerned with the tail-behaviour of the first time that the dynamic Boolean model detects a target particle that is initially placed at the origin $o\in \R^d$. More precisely, we are interested in the {\em detection time}, defined as
\begin{equation}\label{eq:dettime}
T_\mathrm{det}^g:=\inf\big\{t\ge 0\colon g(t)\in\G_t\big\}=\inf\big\{t\ge 0\colon |g(t)-X^i_t|\le R_i\text{ for some }i\ge 1\big\},
\end{equation}
where $g\colon [0,\infty)\to\R^d$ with $g(0)=o$ is some measurable function that represents the movement of the node that is to be detected.

In order to describe the exponential rate of decay of the tail probability of $T_\mathrm{det}^g$, consider $\Gamma_t\subset\R^d$, the {\em trajectory} of the Lévy process $X$ up to time $t$ and shifted by $g$, that is,
\[
	\Gamma^g_t:=\bigcup_{s\leq t}\{X_s + g(s)\},
\]
where we set $\Gamma_t := \Gamma^o_t$. Furthermore, consider {\em $r$-annuli} around Lebesgue-measurable sets $A\subset\R^d$, defined as
\[
	B_r(A):=\big\{x\in\R^d\colon |x-a|<r\text{ for some }a\in A\big\},
\]
writing $|A|$ for the $d$-dimensional Lebesgue-volume of $A\subset\R^d$. 

The probabilistic tail of the detection time can be expressed in terms of the volume of the Lévy sausage up to time $t$, for which we refer to the following asymptotics, cf.~\cite[Theorem~1~\&~2]{Getoor1965}, 
\begin{equation}\label{eq:convergence-volume-capacity}
        \lim_{t\uparrow\infty}\tfrac1t \E|B_r(\Gamma_t)|=r^{d-\alpha}\Cap{(\alpha,d)}, \qquad r\ge 0,
\end{equation}
 where
\[
    \Cap{(\alpha,d)} := 2^\alpha \pi^{d/2}\Gamma(\alpha/2)/\Gamma((d-\alpha)/2)
\]
    with $\Gamma(\cdot)$ being the gamma function. 
    Moreover, the function $t\mapsto \E|B_r(\Gamma_t)| - tr^{d-\alpha}\Cap{(\alpha,d)}$ is monotonically increasing.
        
The following first result features the asymptotics of the upper tails for the distribution of the detection time. We highlight that randomising the initial radii does not change the asymptotic behaviour in a qualitative way. This mainly follows from the fact that the volume asymptotics in~\eqref{eq:convergence-volume-capacity} are continuous in the radius.
\begin{theorem}[Detection time]\label{thm:detection-time}
Consider the dynamic Boolean model with $\alpha\in(0,2]$ and $d>\alpha$.
    \begin{enumerate}[(i)]
        \item 
        	If $g\equiv o$, then
        	\[
        		\P(T_\mathrm{det}^o>t) = \exp\big(-t\lambda \Cap{(\alpha,d)} \E R^{d-\alpha}(1+o(1))\big),
        	\]
        	as $t\to\infty$ and $o(1)\geq0$.
        \item 
        	There exists $C\in(0,1]$ depending only on $\alpha$ and $d$ such that for every measurable $g$
    		\[
    			\P(T_\mathrm{det}^g>t) \leq \exp\big(-Ct\lambda \Cap{(\alpha,d)} \E R^{d-\alpha}\big),
    		\]
            for all $t\ge 0$.
        \item 
        	If \(g\) is the realisation of an independent Lévy process $Y$, then for the annealed probability,
        	\[
        		\lim_{t\uparrow\infty}\tfrac{1}{t}\log\P\big(T_\mathrm{det}^Y>t\big) = \sup_{t\geq 1}\tfrac{1}{t}\log\P\big(T_\mathrm{det}^Y>t\big)\in (-\infty,0),
        	\]
        	i.e., a non-trivial exponential rate exists.
        \item 
        	If $g_\beta(t):=\beta t\psi$ for some $\psi\in\R^d\backslash\{o\}$ and $\beta>0$, then
			\[
				\lim_{\beta\uparrow\infty}\lim_{t\uparrow\infty}\tfrac1t\log\P(T^{g_\beta}_{\rm det}>t)=-\infty,
			\]
        	meaning that a very fast particle is detected fast.
    \end{enumerate}
\end{theorem}
\begin{remark}\label{rem:volume-constant}
 Due to a rearrangement inequality~\cite[Theorem 1.5 \& Corollary 1.7]
{drewitz2014symmetric}, the statement in Part $(ii)$ holds for $C=1$ in the case where $g$ is such that, for all $s\ge 0$, there exists $\delta>0$ such that $\sup_{s\le t<s+\delta}|g(s)-g(t)|<\delta$. In particular, this holds if $g$ is càdlàg.    
\end{remark}
We give the proof of the theorem in \Cref{subsec:proof-detection-time}, which rests on the following first key input. 
Let $T_\mathrm{det}^g(K)$ be the random time until some point of a set $K\subset\R^d$ is detected, i.e.,
    $$T_\mathrm{det}^g(K) := \inf\big\{t\geq0 \colon (K+g(t))\cap \G_t \neq \emptyset \big\},
    $$
    where, for two sets $A,B\subset\R^d$, we consider the \emph{Minkowski sum}
    $$A+B := \{a+b \colon a\in A,\,b\in B\}.$$
 Note that $T_\mathrm{det}^g=T_\mathrm{det}^g(\{o\})$. Then, detection-time probabilities are Poisson void probabilities as stated in the following lemma that we prove in \Cref{subsec:proof-detection-time}. 
\begin{lemma}\label{lem:detection-time-is-void-probability}
    Let $K\subset\R^d$ be a compact set, which will move according to some deterministic measurable motion $g\colon [0,\infty)\to\R^d$.
    Then,
    \[
    \P(T_\mathrm{det}^g(K)>t) = \exp(-\lambda \E|B_R(\Gamma^g_t)+K|).
    \]
\end{lemma}

\subsection{Coverage time}
Instead of detecting a single potentially moving particle, we now consider a static volume $A\subset\R^d$ and ask how long it takes for the process to have ``seen'' all of $A$. Put precisely, we are interested in
\[
	T_\mathrm{cov}(A):= \inf\Big\{t\geq0 \colon A\subset \bigcup_{s\le t}\G_s\Big\}.
\]
While this question is difficult to answer for individual $A$, we can give asymptotics on $T_\mathrm{cov}(kA)$ for large $k$ in terms of its Minkowski dimension $\beta>0$, that is,
\[
	\beta := \lim_{\eps\downarrow 0} \log M(A,\eps)/\log(\eps^{-1}),
\]
if the limit exists and where $M(A,\eps)$ is the minimal number of balls of radius $\eps$ needed to cover $A$. We have the following result, proved in \Cref{subsec:proof-coverage-time}.
\begin{theorem}[Coverage time]\label{thm:coverage-time}Consider the dynamic Boolean model with $\alpha\in(0,2]$ and $d>\alpha$. Then, for $A\subset \R^d$ with Minkowski dimension $\beta$,
    $$\lim_{k\uparrow\infty} \frac{\E T_\mathrm{cov}(kA)}{\log(k)} = \frac{\beta}{\lambda\Cap{(\alpha,d)}\E R^{d-\alpha}}.$$ 
\end{theorem}

\subsection{Percolation time}
Finally, we consider the \emph{percolation time} of a target initially located at the origin and that may or may not move. While the detection time deals with the question when the target is detected for the first time by \emph{any} node, the percolation time requires additionally that the detecting node belongs to the \emph{unbounded component} at the time of detection. 

Recall that for \(d\geq 2\) there exists a critical intensity \(\lambda_c\) such that for all \(\lambda>\lambda_c\), the set
\(
    \G_0
    %= \bigcup_{i\ge 1} B_{R_i}(X_0^i)
\)
contains an unique unbounded connected component almost surely. As mentioned above, by the displacement theorem for Poisson point processes, at each time \(t\), the set of node locations \(\Psi_t:=\{X^i_t\}_{i\ge 1}\) has the same distribution as \(\Psi_0\) so that each \(\Psi_t\) contains such a component almost surely, and we write \(\Psi_t^\infty\subset\Psi_t\) for the set of all nodes that belong to the unbounded component, and write
\[
\G_t^\infty:=\bigcup_{X^i_t\in\Psi_t^\infty}B_{R_i}(X_t^i)
\] 
for the component itself. The percolation time is then defined as
\[
    T_\perc^g := \inf\big\{t\geq 0\colon \exists X^i_t \in \Psi_t^\infty \text{ such that }|X^i_t-g(t)|<R_i\big\} = \inf\big\{t\geq 0\colon g(t)\in \G_t^\infty\big\},
\]
where the function \(g\colon[0,\infty)\to\R^d\), \(g(0)=o\), describes the movement of the target. The next result gives upper tail-bounds for the detection time of a target that does not move too quickly; particularly including the case where the origin does not move at all. For two functions \(f\geq 0\) and \(h>0\), we recall the Landau notation \(f=\omega(h)\) for \(\liminf_{t\to\infty}f(t)/h(t)=\infty\) and the notation \(f=o(h)\) for \(\limsup_{t\to\infty}f(t)/h(t)=0\). 
\begin{theorem}[Percolation time]\label{thm:percTime}
Consider the dynamic Boolean model with \(\alpha\in(0,2)\) and  \(d\geq 2\). 
For any \(\lambda>\lambda_c\) and function \(f(t)=\omega(1)\), there exists a constant \(c=c(d,\lambda,\alpha,f)>0\) such that for all $g$ with \(|g(t)|=\exp(o(t/f(t))\), we have, for all sufficiently large \(t\),
\[
	\P(T_{\perc}^g\geq t) \leq \exp\big(-c t / f(t)\big).
\] 
\end{theorem}

\begin{remark}
    The proof of \Cref{thm:percTime} relies on the decoupling result in \Cref{prop:decoupling} that allows to independently resample the Poisson points in a large box for a positive fraction of time steps, see also \Cref{sec:outlook} for a more detailed description of the underlying heuristics. The decoupling result itself depends on \emph{heat-kernel bounds} and related results for the \(\alpha\)-stable process and is thus only formulated for \(\alpha<2\), cf.\ \Cref{sec:heat-kernel}. However, the result can be extended, \emph{mutatis mutandis}, to the \(\alpha=2\) case by applying the Gaussian heat-kernel bounds instead. By doing so, \Cref{thm:percTime} extends to \(\alpha=2\) and thereby improves the analogous result~\cite[Theorem~1.6]{Peres2011} for the Brownian model with fixed radii in two ways: First, by applying to random radii; second, improving the correction term in the exponential tail from poly-logarithm to any arbitrarily-slow growing function \(f\).   
\end{remark}

\subsection{Discussion and outlook}\label{sec:outlook}
Regarding detection times, one may speculate that the $\beta$-dependent constant in the exponential tail of the probability in \Cref{thm:detection-time} Part $(iv)$ is increasing in $\beta$. This means, at least for linear motions, that the detection becomes faster if the target particle moves faster. Furthermore, note that detection times can be seen as hard-core variants of {\em trapping problems} as considered in~\cite{drewitz2014symmetric}. In this class of models, the covered volume $|B_R(\Gamma_t^g)|$ is replaced by a more general energy term.  
Furthermore, in view of peer-to-peer ad-hoc networks, it is interesting to study the first time a target particle is detected for an uninterrupted time interval $[t,t+a]$, for $a\ge 0$, reflecting the requirement that any transmission takes a minimal amount of time. Also, we suspect a qualitatively similar picture if, instead of considering the detection time by a single node, we study the first time a target is detected by at least $k\ge 2$ nodes, maybe even simultaneously. 
Similar extensions can also be considered in view of the coverage time, e.g., a minimal time of (uninterrupted) coverage or coverage by multiple nodes. Also, what can be said about characteristics beyond the expected value, such as the variance or large deviations?

The proof for the percolation time is based on the following strategy. For \(d\ge 2\) and a Poisson intensity \(\lambda>\lambda_c\), for each time, a typical node is part of the unbounded component with probability \(\theta=\theta(\lambda)>0\). In view of the percolation time with $g\equiv o$ fixed, on an intuitive level, one may expect the following: At the detection time \(T_{\det}\), the origin has been detected by a node of the unbounded component with probability \(\theta\). If this is the case, then \(T_{\perc}=T_{\det}\), and if not, one restarts the processes to check again if the origin has been detected after some time has passed. Assuming enough independence, the number of trials follows a geometric distribution, which ultimately leads to an exponential tail of \(T_{\perc}\) thanks to \Cref{thm:detection-time}. However, this ignores the fact that the model is positively correlated. Indeed, a failed first trial increases the probability of failing again as the unbounded component is, in that case, probably farther away from the origin than on average.
In order to tackle this issue, we establish a result that ensures that the model mixes fast enough so that it behaves like an independent, slightly thinned copy of itself at most integer time steps. This was first shown for a system of Brownian motions with a deterministic radius in \cite{Peres2011}.
The price we pay using this approach is a correction $f$ in the tail estimate of the percolation time, which may be removable, but only with the help of new proof ideas. In view of \Cref{rem:volume-constant}, stating that a large class of moving targets is detected not slower than a static target, it seems reasonable to believe that this is also true for percolation times. The general idea would be that a moving target should lead to faster decoupling due to additional spatial decorrelations. Indeed, one might conjecture that for a fast and linearly moving target \(g_\beta\), as in \Cref{thm:detection-time} Part $(iv)$, we have
\[
\lim_{\beta\uparrow\infty}\lim_{t\uparrow\infty}\tfrac{1}{t}\log\P(T_{\perc}^{g_\beta}>t)=\log(1-\theta).
\]
However, proving this also requires substantial additional effort, which we leave for future work. We also remark that the condition on the displacement of $g$ in \Cref{thm:percTime} ensures appropriately comparable time and space scales. Relevant examples of target nodes, such as static or independent copies of the underlying particle movement, are covered by this condition.

A further characteristic to be studied is the {\em isolation time}, i.e., 
\begin{equation*}
T_\mathrm{iso}^g:=\inf\big\{t\ge 0\colon g(t)\not\in\G_t\big\},
\end{equation*}
as studied in~\cite{peres2013isolation} for the case of Brownian motion. We anticipate an exponential tail behaviour similarly as for transient Brownian motions.

On a general level, one may wonder how the system behaves if, instead of $\Gamma_t$, one would consider Lévy sausages based on a continuously interpolated path of the Lévy processes. More generally, other mobility schemes such as random waypoint models, linear motions or jump-diffusion processes potentially yield different qualitative behaviours. Let us also mention that recently detection times have also been analysed in hyperbolic space~\cite{kiwi2025tail}, opening the door towards the investigation of further network characteristics in this setting.

Finally, we note that the mobile network based on Lévy sausages can serve as a relevant dynamic graph model on which further dynamic processes can be studied. As a prime example, we think of the contact process, as already investigated in dynamic lattice settings, see for example~\cite{seiler2023contact,schapira2025contact,kesten2006phase, deshayes2026contact}. We conjecture that phase transitions similar to the ones shown in \cite{kesten2006phase,Baldasso2023} can be shown, but leave this for future work.

\section{Proofs}\label{sec:proofs}
\subsection{Proofs for detection times}\label{subsec:proof-detection-time}
We start by proving the key \Cref{lem:detection-time-is-void-probability}, which describes detection times as void probabilities. 
\begin{proof}[Proof of \Cref{lem:detection-time-is-void-probability}]
    We adapt the proof given in~\cite{Peres2011}. 
    Recall that $\Psi_0\subset\R^d$ denotes the Poisson point process of intensity $\lambda>0$ representing the initial node locations. Let $\Psi\subset\Psi_0$ denote the set of nodes that detect any point in the (moving) compact set $K$ up to time $t\ge 0$. As each node moves and possesses a random detection range,  independently of the others, $\Psi$ is an independent Poisson thinning and thus an inhomogeneous Poisson point process whose intensity measure has Lebesgue density given by
    \[
    \begin{aligned}
&\rho(x):=\lambda \P\big(\exists s\leq t \colon K+g(s)\cap B_R(X_s+x) \neq \emptyset  \big)\\
    &\qquad = \lambda \P\big(\exists s\leq t \colon x \in B_R(K+g(s)-X_s)  \big)
    = \lambda \P\big( x\in B_R(\Gamma_t^g) + K  \big),
    \end{aligned}
    \]
    where we have also used that $X_s$ has the same distribution as $-X_s$ in the last equality.
    Thus, the probability of being undetected up to time $t$ is the void probability
      \[ 
        \begin{aligned}
            \P(T_\mathrm{det}^g(K)>t) 
            &
                = \P(\Psi=\emptyset) 
                = \exp\Big(-\int_{\R^d} \lambda \P\big( x\in B_R(\Gamma_t^g) + K  \big) \dx x\Big) 
            \\ &
                = \exp\big(- \lambda \E|B_R(\Gamma_t^g) + K|\big),
        \end{aligned}
        \]
    concluding the proof.
\end{proof}

From now on, we will mainly concern ourselves with volume estimates to derive statements about detection probabilities. The following lemma will be helpful. We give the proof later in this section.
\begin{lemma}[Volume bounds]\label{lem:Boolean-Model-under-shrinkage}
    Let $A\subset\R^d$ be a set, then $|B_r(A)|\leq r^d|B_1(A)|$.
\end{lemma}
Returning to volume asymptotics, recall that the radius $R$ has finite $d$-th moment.
\begin{lemma}[Volume asymptotics for random radii]\label{lem:volume-asymptotic-random-radii}
We have for all $t\ge 0$
    \[
        \E|B_R(\Gamma_t)| \geq t\Cap{(\alpha,d)} \E R^{d-\alpha}\qquad\text{ and }\qquad\lim_{t\uparrow\infty}\tfrac1t \E|B_R(\Gamma_t)| = \Cap{(\alpha,d)} \E R^{d-\alpha}.
    \]
\end{lemma}
\begin{proof}
    Let $M>0$. Then, by dominated convergence with majorant $\E|B_M(\Gamma_t)|$ and~\cite[Theorem~1~\&~2]{Getoor1965},
    \[\begin{aligned}
        \lim_{t\uparrow\infty} \tfrac1t \E|B_R(\Gamma_t)|& %=\lim_{t\uparrow\infty}\tfrac1t\E\big[\E|B_R(\Gamma_t)|\big]\\
        \geq \lim_{t\uparrow\infty} \tfrac1t \E\big[\E|B_R(\Gamma_t)|\I\{R<M\}\big]=\Cap{(\alpha,d)}\E\big[R^{d-\alpha}\I\{R<M\}\big]. 
    \end{aligned}
    \]
    Thus, with $M\uparrow\infty$, we have
    $\lim_{t\uparrow\infty} \tfrac1t \E|B_R(\Gamma_t)| \geq \Cap{(\alpha,d)} \E R^{d-\alpha}$.
    On the other hand, with the same arguments together with \Cref{lem:Boolean-Model-under-shrinkage}, we have
    \[\begin{aligned}
        \lim_{t\uparrow\infty} \tfrac1t \E|B_R(\Gamma_t)| &= \lim_{t\uparrow\infty} \tfrac1t \E\big[\E|B_R(\Gamma_t)|\I\{R<M\}\big]+  \lim_{t\uparrow\infty}\tfrac1t\E\big[\E|B_R(\Gamma_t)|\I\{R\ge M\}\big]\\
        &\leq \lim_{t\uparrow\infty} \tfrac1t \E\big[\E|B_R(\Gamma_t)|\I\{R<M\}\big]+  \lim_{t\uparrow\infty}\tfrac1t \E|B_1(\Gamma_t)| \E\big[R^d\I\{R\ge M\}\big]\\
        &= \Cap{(\alpha,d)}\E\big[R^{d-\alpha}\I\{R<M\}\big]+ \Cap{(\alpha,d)} \E\big[R^d\I\{R\ge M\}\big]<\infty.
    \end{aligned}
    \]
    So again, as $M\uparrow\infty$, we have $\lim_{t\uparrow\infty} \tfrac1t \E|B_R(\Gamma_t)| \leq \Cap{(\alpha,d)} \E R^{d-\alpha}$.
    The first statement follows from the monotonicity statement below~\eqref{eq:convergence-volume-capacity}.
\end{proof}

Let us turn our attention to moving particles. We have the following lower volume bound.

\begin{lemma}[Universal lower volume bound for drifts]\label{lem_PSSS_2.4.}
    There exists $C\in(0,1]$ depending only on $\alpha$ and $d$ such that
    \[\E|B_R(\Gamma^g_t)| \geq C \E|B_R(\Gamma_t)| \ge  C t\Cap{(\alpha,d)}\E R^{d-\alpha}.\]
\end{lemma}
In the case of Brownian motions $\alpha=2$, \cite{peres2012isoperimetric} shows that $C=1$ even for non-measurable $g$. As mentioned in \Cref{rem:volume-constant}, $C=1$ is only known under mild assumptions on $g$ for $\alpha$-stable processes.
\begin{proof}
    It suffices to check the statement for fixed deterministic $R=r>0$. First, we write
    \[Z^t_y := \int_0^t \I \big\{ X_s + g(s) \in B_r(y)\big\}\d s\]
    and note that we have
    $$\P\big(Z^t_y > 0\big) = \P\big(\exists s\leq t\colon  X_s + g(s) \in B_r(y)\big).$$
    Thus, we have
    \[\E\big|B_r(\Gamma^g_t)\big| = \int_{\R^d}\P\big(\Gamma^g_t \cap B_r(y)\neq\emptyset\big)\dx y = 
\int_{\R^d}\P\big(\exists s\leq t\colon X_s + g(s) \in B_r(y)\big)\dx y = \int_{\R^d}\P\big(Z^t_y >0\big) \dx y
.\]
Note that for general random variables $Y\geq0$ with $\P(Y>0)>0$, we have
$$\P(Y>0) = \E Y /\E[Y \,\vert\, Y>0].$$
Let us estimate the denominator. Denote by $\tau$ the first hitting time of (the closure of) $B_r(y)$ for $\Gamma_t^g$. The conditioning $Z_y>0$ is thus equivalent to $\tau<t$, which in turn implies $X_\tau+g(\tau)\in B_r(y)$. 
Let $Y=\{Y_s\}_{s\geq0}$ be an independent copy of $X$. Then, 
    \[\begin{aligned}
        \E\big[Z^t_y \,\vert\, Z^t_y>0\big] &= \E\Big[ \int_0^t \I \big\{ X_s + g(s) \in B_r(y)\big\} \dx s \,\Big\vert\, \tau < t \Big]\\
        &= \E\Big[\int_0^t \I\big\{ X_{s+\tau} + g(s+\tau) \in B_r(y)\big\} \dx s \,\Big\vert\, \tau < t \Big]\\
        &= \E\Big[ \int_0^{t-\tau} \I\big\{ X_{s+\tau} + g(s+\tau) \in B_r(y)\big\} \dx s \,\Big\vert\, \tau < t \Big]\\
        &= \E\Big[\int_0^{t-\tau}\I\big\{ X_{s+\tau}-X_\tau+g(s+\tau)-g(\tau) + X_\tau+g(\tau)\in B_r(y)\big\} \dx s\,\Big\vert\, \tau < t \Big].
        \end{aligned}
    \]
        By the strong Markov property, the right-hand side can be upper bounded by
        \[\begin{aligned}
          \E\Big[&\int_0^{t-\tau}\I\big\{ Y_{s} + g(s+\tau)-g(\tau)\in B_{2r}(o)\big\} \dx s \,\Big\vert\, \tau < t \Big]\\
        &=\E\Big[ \int_0^{t-\tau} \P\big(Y_{s} + g(s+\tau)-g(\tau)\in B_{2r}(o)\big) \dx s\,\Big\vert\, \tau < t \Big]\\
        &\leq\E\Big[ \int_0^{t-\tau} \P\big(Y_{s}\in B_{2r}(o)\big) \dx s\,\Big\vert\, \tau < t \Big]\\
        &\leq\E\Big[ \int_0^{\infty} \P\big(Y_{s}\in B_{2r}(o)\big) \dx s\,\Big\vert\, \tau < t \Big]\\
        &=\int_0^{\infty} \P\big(Y_{s}\in B_{2r}(o)\big) \dx s= r^\alpha\int_0^{\infty} \P\big(Y_{s}\in B_{2}(o)\big) \dx s
        = r^\alpha C <\infty,
    \end{aligned}\]
due to transience of the $\alpha$-stable process and where we used~\eqref{eq:self-similarity-alpha-stable} in the second-to-last equality.
On the other hand,
    \[\begin{aligned}
        \int_{\R^d} \E Z^t_y \ \dx y &= \int_{\R^d} \E\Big[\int_0^t \I\big\{ X_s + g(s) \in B_r(y)\big\} \dx s\Big] \dx y= \int_{\R^d} \int_0^t \P\big( X_s + g(s) \in B_r(y)\big) \dx s \, \dx y\\
        &= \int_{\R^d} \int_0^t \int_{B_r(o)} f_s\big(z+y-g(s)\big) \dx z \, \dx s\, \dx y= \int_0^t \int_{B_r(o)} 1 \dx s\, \dx y = t|B_r(o)|.
    \end{aligned}\]
Putting everything together, we have $\E|B_r(\Gamma^g_t)| \geq r^{-\alpha}t|B_r(o)|/C$, which finishes the proof.
\end{proof}

\begin{proof}[Proof of \Cref{thm:detection-time}]
    Items~(i) and (ii) are direct consequences of the \Cref{lem:volume-asymptotic-random-radii,lem_PSSS_2.4.}  after applying \Cref{lem:detection-time-is-void-probability}. For
    Item~(iii) note that for all $t,t'\geq 0$ we have that
        \[\begin{aligned}
\P\big(T_\mathrm{det}^Y>t+t'\big) &= \P\big(o\notin B_R(\Gamma_{t+t'}^Y)\big)
            = \E_Y\big[ \exp(-\lambda \E_X|B_R(\Gamma_{t+t'}^Y)|) \big]\\
            &= \E_Y\Big[\exp\Big(-\lambda \E_X\Big|\bigcup_{s\leq t+t'}B_R(X_s+Y_s)\Big|\Big)\Big]\\
            &\geq \E_Y\Big[\exp\Big(-\lambda \E_X\Big[\Big|\bigcup_{s\leq t}B_R(X_s+Y_s)\Big| + \Big|\bigcup_{t\leq s\leq t+t'}B_R(X_s+Y_s)\Big|\Big]\Big)\Big]\\
            &= \E_Y\Big[\exp\Big(-\lambda \E_X\Big|\bigcup_{s\leq t}B_R(X_s+Y_s)\Big|\Big)\Big]\E_Y\Big[\exp\Big(-\lambda\E_X\Big|\bigcup_{s\leq t'}B_R(X_s+Y_s)\Big|\Big)\Big]\\
            &= \P\big(T_\mathrm{det}^Y>t\big) \P\big(T_\mathrm{det}^Y>t'\big),
        \end{aligned}
        \]
        where we used independent increments and shift-invariance in space and time in the fourth line. 
    Thus, $-\log \P(T_\mathrm{det}^Y>t)$ is subadditive, so $\lim_{t\uparrow\infty}\tfrac{1}{t}\log\P(T_\mathrm{det}^Y>t) = \sup_{t\geq 1}\tfrac{1}{t}\log\P(T_\mathrm{det}^Y>t)$ by Fekete's Lemma. It remains to show that this limit is nontrivial. Item~(ii) gives an upper bound of $-C\lambda\Cap(\alpha,d)\E R^{d-\alpha}$. With regards to the lower bound, it suffices to check $\P(T_\mathrm{det}^Y>1)>0$. Since $Y$ is a Lévy process and in particular càdlàg, there exists some $M>0$ with $\P(Y_{[0,1]}\subset B_M(o))=:\eps>0$. Consider the dynamic Boolean model with the radii enlarged by $M$, i.e., $\tilde{\G}_t:=\cup_{i\geq1}B_{R_i+M}(X_t^i)=B_M(\G_t)$. Then, the detection time $\tilde{T}_\mathrm{det}^o$ for this model satisfies, by Item~(i),     \[\P(\tilde{T}_\mathrm{det}^o>1) = \exp\big(-\lambda\Cap(\alpha,d)\E[(R+M)^{d-\alpha}](1+o(1))\big)=:\delta>0.\]
    Since $Y$ is independent of $X_i$ and $R_i$, we have
    \[\P\Big(o\notin\bigcup_{s\leq1}B_M(\G_s) \text{ and }Y_{[0,1]}\subset B_M(o)\Big)=\P\big(\tilde{T}_\mathrm{det}^o>1 \text{ and } Y_{[0,1]}\subset B_M(o)\big) = \eps\delta>0.\]
    However, this event implies $\{Y_s\notin\G_s\, \forall s\leq1\}=\{T_\mathrm{det}^Y>1\}$, proving the claim.
    
    For Item~(iv), it suffices to show that 
\(\lim_{\beta\uparrow\infty}\lim_{t\uparrow\infty}\tfrac1t\lim_{\beta\uparrow\infty}\E\big|\bigcup_{s\le t}B_R\big(X_s-g_\beta(s)\big)\big|=\infty.\) Inspecting the proof of \Cref{lem_PSSS_2.4.}, we see that it suffices to show that 
\[\lim_{\beta\uparrow\infty}\lim_{t\uparrow\infty}\sup_{y\in \R^d}\E\big[Z^{t,\beta}_y\big|Z^{t,\beta}_y>0\big]=0.\] 
For this, we can follow the initial steps as above, and bound
\[\begin{aligned}
\lim_{\beta\uparrow\infty}\lim_{t\uparrow\infty}\sup_{y\in \R^d}\E\big[Z^{t,\beta}_y\big|Z^{t,\beta}_y>0\big]&\le\lim_{\beta\uparrow\infty}\int_0^\infty\P\big(Y_{s} - g_\beta(s)\in B_{2R}(o)\big) \dx s\\
&=\int_0^\infty \lim_{\beta\uparrow\infty}\P\big(Y_{s}\in B_{2R}(g_\beta(s))\big) \dx s\\
&= \int_0^\infty \lim_{\beta\uparrow\infty}\P\big(s^{1/\alpha}Y_{1}\in B_{2R}(g_\beta(s))\big) \dx s\\
&= \int_0^\infty \lim_{\beta\uparrow\infty}\P\big(Y_{1}\in B_{2R/s^{1/\alpha}}(\beta s^{1-1/\alpha}\psi)\big) \dx s = 0,
\end{aligned}
\]
by dominated convergence with majorant $\I{\{s\leq1\}}+\P\big(Y_{s}\in B_{2R}(o)\big)$, which is integrable by transience of the process. In the last line we again used~\eqref{eq:self-similarity-alpha-stable}.
\end{proof}

Let us finish the section with the proof of \Cref{lem:Boolean-Model-under-shrinkage}. To do so, we first show two auxiliary results. %Recall that $o\in\R^d$ denotes the origin.
\begin{lemma}
    Let $x\in\R^d$, $0 \leq r_1 \leq r_2$ and $c\in [0,1]$. Then,
    \begin{equation} \label{eq:overlap-monotonicity}
        B_{r_1}(o)\cap B_{r_2}(x)\subset B_{r_1}(o)\cap B_{r_2}(cx).
    \end{equation}
\end{lemma}
\begin{proof}
     Let $y\in B_{r_1}(o)\cap B_{r_2}(x)$. Then, $\|y\|< r_1$ and $\|x-y\|< r_2$. By convexity of the norm and \(r_1\leq r_2\), we have
    \[
        D(c)=\|cx - y \| \leq c\|x-y\| + (1-c)\|y\| < c r_2 + (1-c)r_1 \leq r_2.
    \]
    Thus, $y\in B_{r_2}(cx)$, proving the claim.
\end{proof}

\begin{lemma}\label{lem:shrinking-boolean-model}
    Let $A\subset\R^d$ be a set and $c\in [0,1]$. Then, 
    $|B_r(A)|\geq |B_r(cA)|$.
\end{lemma}
\begin{proof}
    Let us first consider the case where $A$ consists of finitely many points, i.e., we have 
    $A=\{x_0,x_1,\dots, x_n\}$. We prove the claim by induction. The claim is trivial for singletons. After shifting, we may, without loss of generality, assume $x_0=o$. Write $A_1=A\backslash \{o\}$. Then, $|B_r(cA_1)|\leq |B_r(A_1)|$. Furthermore,
    \begin{align*}
    |B_r(A)| &= |B_r(o)| + |B_r(A_1)| - |B_r(o)\cap B_r(A_1)| \text{ and }\\
    |B_r(cA)| &= |B_r(o)| + |B_r(cA_1)| - |B_r(o)\cap B_r(cA_1)|.
    \end{align*}
    Thus, we only need to show that
    $$B_r(o)\cap B_r(A_1) \subset B_r(o)\cap B_r(cA_1).$$
    But this follows from \eqref{eq:overlap-monotonicity} via
    $$ B_r(o)\cap B_r(A_1) = \bigcup_{i=1}^n B_r(o)\cap B_r(x_i) \subset \bigcup_{i=1}^n B_r(o)\cap B_r(cx_i) = B_r(o)\cap B_r(cA_1)$$
    and so the claim holds true in the finite case.
    
    Let us move on to arbitrary $A\subset\R^d$. First, we may assume that $A$ is bounded, since otherwise $|B_r(A)| = |B_r(cA)|=\infty$. 
    Let $\eps>0$. Then, we find a finite set $I(\eps)\subset A$ such that 
    $$B_{r-\eps}(A)\subset B_r(I(\eps)),$$
    since the closure of $B_{r-\eps}(A)$ is compact and covered by $\{B_r(x)\}_{x\in A}$. We show 
    $$B_{r-\eps}(cA) \subset B_r(cI(\eps)).$$
    This will yield the claim since
    $$B_r(cA)\nwarrow_{\eps\downarrow 0} B_{r-\eps}(cA) \subset B_r(cI(\eps)) \quad\text{ and }\quad
    B_r(I(\eps)) \subset B_r(A),$$
    as well as $|B_r(cI(\eps))| \leq |B_r(I(\eps))|$ since $I(\eps)$ is finite. Let $a\in B_{r-\eps}(cA)$. Hence, we find some $x\in A$ such that $|a-cx|<r-\eps$. On the other hand, $|a+(1-c)x-x|=|a-cx|<r-\eps$ and thus $a+(1-c)\in B_{r-\eps}(A)$. Therefore, $a+(1-c)x\in B_{r}(y)$ for some $y\in I(\eps)$. As such, we need to show $a\in B_r(cy)$. Using~\eqref{eq:overlap-monotonicity}, we have 
    $$a-cx\in B_{r-\eps}(o)\cap B_r(y-x)\subset 
    B_{r-\eps}(o)\cap B_r(cy-cx),$$
    which implies $r>\|a-cx-cy+cy\|=\|a-cy\|$, i.e., $a\in B_r(cy)$. This finishes the proof.
\end{proof}
\begin{proof}[Proof of \Cref{lem:Boolean-Model-under-shrinkage}]
    The claim follows from \Cref{lem:shrinking-boolean-model}, which yields $|B_1(r^{-1}A)| \leq |B_1(A)|$. Thus, 
    \[|B_r(A)|=|r\cdot B_1(r^{-1}A)|=r^d |B_1(r^{-1}A)| \leq r^d|B_1(A)|,\]
   as desired.
\end{proof}

\subsection{Prerequisite results on $\alpha$-stable processes}
The isotropic $\alpha$-stable process has the Lévy intensity function (or jump activity)
\begin{equation}\label{eq:jump-kernel}
    \mathcal{J}(x,y) := C|x-y|^{-(d+\alpha)},\quad x,y\in\R^d,
\end{equation}
for some $C>0$.

\subsubsection{Heat-kernel bounds, parabolic Harnack inequality and  H\"older inequality} \label{sec:heat-kernel}

The following statement is an application of \cite[Theorem 1.2]{Chen2008} to the specific case of isotropic $\alpha$-stable processes on $\mathbb{R}^d$.
\begin{theorem}[Heat-kernel bounds]\label{thrm:heatkernel}
  The transition probability at time $t\ge0$ of an isotropic  $\alpha$-stable process on $\mathbb{R}^d$ with $\alpha\in (0,2]$ has a Lebesgue density $p(t,\cdot,\cdot)\colon  \R^d\times \R^d\to [0,1]$ such that there exist constants $c_1,c_2,C>0$ with
    \begin{align*}
        C^{-1}\Big(t^{-d/\alpha}\wedge c_1\frac{t}{|x-y|^{d+\alpha}}\Big)\leq p(t,x,y)
        \leq C\Big(t^{-d/\alpha}\wedge c_2\frac{t}{|x-y|^{d+\alpha}}\Big),\qquad\text{ for all }     t>0, x,y\in\R^d.
    \end{align*}

\end{theorem}
\begin{proof}
    In the notation of \cite{Chen2008}, we choose $V(r)=r^d$ and $\phi(r)=r^\alpha$. 
    Consequently, with $\rho$ being the Euclidean metric on $F=\R^d$, the jump activity satisfies Condition (1.9) from \cite{Chen2008}, i.e.,
    \[
        \mathcal{J}(x,y)\asymp \frac{1}{V(\rho(x,y))\phi(\rho(x,y))}.
    \]
    Then, \cite[Condition (1.1)]{Chen2008} is satisfied for the Euclidean metric on $\R^d$, as is Condition (1.8) that $\mu(B(x,r))\asymp V(r)$ for $x\in\R^d$ and $r>0$ (here, $\mu$ is the Lebesgue measure on $\R^d$). Condition (1.11) is also trivially satisfied for $V(r)=r^d$. In our case, $\phi(r)=r^\alpha$, so $\phi_1(r)=r^\alpha$ and $\psi(r)=1$, so in the notation of \cite{Chen2008}, Condition (1.12) is satisfied with $\gamma_1=\gamma_2=0$ and by \cite[Theorem 1.2]{Chen2008}, the heat-kernel bounds obtained hold for all $t>0$ and all $x,y\in\R^d$.
\end{proof}
We say that a non-negative Borel measurable function $h(t,x)$ on $[0,\infty)\times \R^d$ is {\em parabolic} in a relatively open subset $D$ of $[0,\infty)\times \R^d$ if, for every relatively compact open subset $D_1$ of $D$, $h(t,x)=\E^{(t,x)}[h(Z_{\tau_{D_1}})]$ for every $(t,x)\in D_1$, where $Z_s=(V_s,X_s)$ is the space-time process with $V_s=V_0+s$ for some $V_0\in\R$ and $\tau_{D_1}=\inf\{s>0\colon Z_s\not\in D_1\}$.  Note in particular that the transition density of the isotropic $\alpha$-stable process $p(t,x,y)$ is parabolic as a function of $t$ and $y$.

 For each $r,t>0$, $x\in \R^d$ and $\gamma\in(0,1/2)$ define the cylinder
\[
Q(t,x,r):=[t,t+\gamma r^\alpha]\times B_r(x).
\]
\begin{theorem}[Parabolic Harnack inequality]
   Consider the heat kernel $p$ as before. Then, there exists $c_1>0$ such that for every $z\in \R^d$, $r>0$ and every non-negative function $h$ on $[0,\infty)\times\R^d$ that is parabolic and bounded on $Q(0,x,2r)$, we have
    \[
        \sup_{(t,y)\in Q(r^\alpha,z,r)}h(t,y)\leq c_1\inf_{y\in B_r(z)}h(0,y).
    \]
    In particular, for $t>0$,
    \[
        \sup_{(s,y)\in Q((1-\gamma)t,z,t^{1/\alpha})}p(s,x,y)\leq c\inf_{y\in B_{t^{{1}/{\alpha}}}(z)}p((1+\gamma)t,x,y).
    \]
\end{theorem}

\begin{proof}
    The statement is \cite[Theorem 4.12]{Chen2008} applied to our setting of isotropic $\alpha$-stable processes (as depicted in the proof of \Cref{thrm:heatkernel} with $\delta=1$).
\end{proof}

For our purposes, the key consequence of the parabolic Harnack inequality is that the heat kernel of an isotropic $\alpha$-stable process is H\"older continuous.

\begin{prop}[Hölder continuity]\label{prop:hoelder}
    For every $r_0>0$ there exist constants $c,\kappa>0$ such that for every $0<r<r_0$ and every bounded parabolic function $h$ on $Q(0,x_0,2r)$,
    \[
        |h(s,x)-h(t,y)|\leq c\|h\|_{\infty,\R^d} r^{-\kappa}(|t-s|^{1/\alpha}+|x-y|)^\kappa
    \]
    holds for $(s,x)$, $(t,y)\in Q(0,x_0,r)$, where $$\|h\|_{\infty,\R^d}:=\sup_{(t,y)\in[0,\gamma (2r)^\alpha]\times \R^d}|h(t,y)|.$$ 
    In particular, for $p$ and any $T>0$ and $t_0\in(0,T)$, there exist constants $c>0$ and $\kappa>0$ such that for any $t,s\in [t_0,T]$ and $x_i,y_i\in \R^d\times \R^d$ with $i=1,2,$
    \[
        |p(s,x_1,y_1)-p(t,x_2,y_2)|\leq c t_0^{-(d+\kappa)/\alpha}(|t-s|^{1/\alpha}+|x_1-x_2|+|y_1-y_2|)^\kappa.
    \]
\end{prop}
\begin{proof}
    Again, this statement is \cite[Proposition 4.14]{Chen2008} applied to our setting of isotropic $\alpha$-stable processes (as depicted in the proof of \Cref{thrm:heatkernel}).
\end{proof}
\subsubsection{Escape and hitting probabilities}
We will need bounds on the probabilities of either escaping some ball of radius $r$ or hitting some far-away ball of radius $r$.
\begin{prop}[Escape probability]\label{prop:exitTime} 
Let $\tau_r:=\inf\{t>0\colon | X_t|\geq r\}$ be the first time the isotropic $\alpha$-stable process with $\alpha<2$ is at least $r$ away from its starting location $o$. Then, there exists $C>0$ such that
    \[
        \P(\tau_{r}\leq t)\le C r^{-\alpha} t,\qquad\text{ for all } r>0,t<r^\alpha.
    \]
\end{prop}

\begin{proof}
We first check that its jump activity $\mathcal{J}$ satisfies
\[
    \int\frac{|z|^2}{1+|z|^2}\mathcal{J}(o,z)\dx z<\infty
\]
for every $\alpha<2$. Our process starts in the origin by assumption, so
\cite[Equation (3.2)]{Pruitt1981} gives for every $t>0, r>0$,
\[
\mathbb{P}(\tau_r\leq t)
=\mathbb{P}\Big(\sup_{s\le t}|X_s|\geq r\Big)
 \leq c_2\min\{1, th(r)\},
\]
where $h(r)=K(r)+L(r)$ with
$K(r)=\int_{|z|\leq r}r^{-2}|z|^2 \mathcal{J}(o,z)\dx z$ and
$L(r)=\int_{|z|>r}\mathcal{J}(o,z)\dx z$ (the third term in \cite{Pruitt1981} vanishes due to rotational invariance).
Thus, the isotropic $\alpha$-stable process gives $h(r)\asymp r^{-\alpha}$.
Hence, there exists $C>0$ such that, uniformly for arbitrary $t$ satisfying $t\le r^{\alpha}$,
\[
\mathbb{P}(\tau_r\le t)\le C r^{-\alpha}t,
\]
as desired.
\end{proof}

The following result estimates how far the $\alpha$-stable process $X$ can travel without large jumps.  For this, we  decompose $X$ into independent processes 
$$X=X'+X'',$$ 
where $X'$ has jump activity $\mathcal{J}(x,y)\I\{|x-y|\geq1\}$ and $X''$ has jump activity $\mathcal{J}(x,y)\I\{|x-y|<1\}$. Thus, $X'$ is a {\em compound Poisson process} of finite activity $\int_{B_1(o)^\complement}\mathcal{J}(o,x)\dx x\in(0,\infty)$ only consisting of jumps of size $\geq1$. 
On the other hand, $X''$ is a pure-jump martingale having only jumps of size less than $1$, i.e., it is the truncated $\alpha$-stable process. 

\begin{lemma}[Escape probability of truncated process]
     Let $r\geq0$ be fixed and $\alpha\in(0,2)$. Then, there are constants $C,C',\kappa>0$, such that
    $$\P\big(\max_{s\leq t}|X''_s|\geq L\big) \leq C \exp\big(-\kappa (L-C't)\big),\qquad \text{ for all } L\geq0, t\geq 1.$$
\end{lemma}
\begin{proof}
     Denote the heat kernel of $X''$ by $p^-(t,y,x)$.
     By~\cite[Theorem~2.3]{ChenKimKumagai2011}, there exist $C^*,c_1,c_2>0$  such that for every $|x|\geq C^*t$, we have 
    $$p^-(t,o,x) \leq c_1 \exp\big(-c_2 |x|\log(|x|/t)\big) \leq c_1 \exp\big(-2 \kappa |x|\big), \quad\text{ where }
    \kappa:=\tfrac{1}{2}c_3\log C^*.$$ 
    On the other hand, we also have for every $t\geq 1$ and every $x\in\R^d$ by~\cite[Proposition~2.2]{ChenKimKumagai2011} that
    $$p^-(t,o,x) \leq c_1 t^{-d/2}.$$
    We use Doob's martingale inequality to bound for all $\kappa>0$
    $$\P\big(\max_{s\leq t}|X''_s|\geq \ell\big) \leq \exp(-\kappa \ell)\E \exp(\kappa|X''_t|).$$
    Thus, with constants $C>0$ only depending on $d,C^*,\kappa$ changing from line to line
        \begin{align*}
            \E\exp(\kappa|X''_t|) &= \int_{\R^d}p^-(t,o,x)\exp(\kappa |x|)\dx x \\
            &\leq \int_{|x|\leq C^*t}c_1 t^{-d/2} \exp(\kappa |x|) \dx x + \int_{|x|\geq C^*t} c_1 \exp(-2\kappa |x|) \exp(\kappa |x|) \dx x\\
            &\leq C t^{-d/2}\int_0^{C^*t} z^{d-1} \exp(\kappa z)\dx z + C \int_{0}^\infty  z^{d-1}\exp(-\kappa z)\dx z\\
            &\leq C t^{d/2-1}\exp(\kappa C^*t) + C \leq C \exp(2\kappa C^*t).
        \end{align*}
    Therefore, we conclude
    $$\P\big(\max_{s\leq t}|X''_s|\geq L\big) \leq \exp(-\kappa L)\E\exp(\kappa|X''_t|) \leq C t^{d/2-1}\exp\big(-\kappa (L-2C^*t)\big),$$
    as desired.
\end{proof}

\begin{lemma}[Hitting probability of far-away sets]\label{lem:hitting-probability}
    There exist $C,C',\kappa>0$ (only depending on $d$ and $\alpha$) such that for every $t\geq 1$, $x\in\R^d$ and $r,L>0$ with $r+L\leq |x|/6$, we have
    $$\P\big(\exists s\leq t\colon  X_s\in B_r(x)\big) \leq C \Big( |x|^{-(d+\alpha)}t^2(r+L)^d + \exp\big(-\kappa (L/2-C't)\big) \Big).$$
    In particular, by choosing $L:=t\log |x|$ and if $|x|\geq t$, we have asymptotically as $t\uparrow\infty$ that
    \[\label{eq:hitting-probability}
    \P\big(\exists s\leq t\colon  X_s\in B_r(x)\big) \leq C |x|^{-d-\alpha+o(1)}t^{2+d}.
    \]
\end{lemma}
\begin{proof}
First, we note that
    \begin{align*}
            \big\{\exists &s\leq t\colon X_s\in B_r(x)\big\}\\
            &            \subset \Big\{\max_{s\leq t}|X''_s|\geq L/2 \Big\} \cup\big\{X \text{ has a jump of size } \geq1 \text{ into } B_{r+L}(x)\text{ before time }t\big\}.
    \end{align*}
    This is due to the fact that if $X_s$ hits $B_r(x)$ without having a jump of size at least $\geq1$ into $B_{r+L}(x)$, then $X$ had to move from $\R^d\backslash B_{r+L}(x)$ to $B_r(x)$ without any jumps of size $\geq1$. This implies that the truncated process $X''$ has two points of distance at least $L$ to each other up to time $t$, which shows the set inclusion of the events.
    The probability of the first event can be estimated using the previous lemma, yielding 
    $$\P\big(\max_{s\leq t}|X''_s|\geq L/2\big) \leq C \exp\big(-\kappa (L/2-C't)\big).$$
    
    For the probability of the second event, we use several steps. First, note that $X$ having a jump of size $\geq1$ is the same as $X'$ performing such a jump.
    The result will follow from the Markov inequality and Markov property of the jump process as follows.

    Let us split the event 
    $\{X \text{ has a jump of size} \geq1 \text{ into } B_{r+L}(x)\}$
    into the cases where the jump has size $\geq |x|/3$ or not.
    If the jump has size $\geq|x|/3$, then with constants $C$ changing from line to line (independent of $t,L,x,r$ but not $d,\alpha$),
        \begin{equation*}
        \begin{aligned}
            \P\big(X &\text{ has jump of size} \geq|x|/3 \text{ into } B_{r+L}(x)\big) 
            \leq \E\big[\#\{\text{jumps of size} \geq|x|/3 \text{ into }B_{r+L}(x)\}\big]\\
            &= \int_0^t \int_{\R^d} \E\big[\#\{\text{jumps of size } \geq|x|/3 \text{ into }B_{r+L}(x)\} \text{ from }y \text{ at time }s\big] \dx y \,\dx s\\
            &= \int_0^t \int_{\R^d}  p(s,o,y) \int_{\R^d}\I\{y+z\in B_{r+L}(x)\}\I\{|z|\geq |x|/3\} \mathcal{J}(y,y+z) \I\{|z|\geq1\}\dx z\,\dx y \,\dx s\\
            &\leq C \int_0^t  \int_{\R^d} p(s,o,y) \int_{\R^d} \I\{y+z\in B_{r+L}(x)\}\I\{|z|\geq|x|/3]\} |z|^{-(d+\alpha)}\dx z\, \dx y \,\dx s \\
            &\leq C |B_{r+L}(o)| (|x|/3)^{-(d+\alpha)}\int_0^t   \int_{\R^d} p(s,o,y) \dx y \,\dx s
            \leq C (r+L)^d |x|^{-(d+\alpha)}t.
        \end{aligned}
    \end{equation*}
    
    On the other hand, if the jump has size $\in[1,|x|/3]$, then $X$ had to be in $B_{|x|/3+r+L}(x)\subset B_{|x|/2}(x)$ prior to the jump. In particular, this has distance $\geq|x|/2$ to the origin.
    By the assumption $r+L\leq |x|/6$, we get $|x|/2 \geq |x|/3 + r + L$. Thus, using the heat-kernel estimates on $X$, \Cref{thrm:heatkernel}, as well as the finite jump rate of $X'$, we have with constants $C$ changing from line to line
        \begin{align*}
            \P\big(X &\text{ has jump of size} \in [1,|x|/3] \text{ into } B_{r+L}(x)\big) 
            \leq \E\big[\#\{\text{jumps of size}\in [1,|x|/3] \text{ into }B_{r+L}(x)\}\big]\\
            &=\int_0^t \int_{\R^d}  p(s,o,y) \nu \int_{\R^d}\I\{y+z\in B_{r+L}(x)\}\I\{1\leq |z|\leq |x|/3\} \mathcal{J}(y,y+z) \I\{|z|\geq1\}\dx z\,\dx y \,\dx s\\
            &=  \int_{\R^d} \I\{1\leq |z|\leq |x|/3\}\mathcal{J}(o,z)\dx z
            \int_0^t\int_{B_{|x|/2}(x)\cap B_{r+L}(z-x)} p(s,o,y) \dx y \,\dx s \\
            & \leq C \int_0^t   \int_{B_{|x|/2}(x)\cap B_{r+L}(z-x)} \frac{s}{|y|^{d+\alpha}} \dx y \,\dx s\\
            &\leq C \int_0^t   \int_{B_{|x|/2}(x)\cap B_{r+L}(z-x)} \frac{t}{(|x|/2)^{d+\alpha}} \dx y \,\dx s\\
            &\leq C t^2 |B_{r+L}(o)|  |x|^{-(d+\alpha)} \leq C (r+L)^d |x|^{-(d+\alpha)}t^2,
        \end{align*}
    with $C$ only depending on $d$ and $\alpha$ as before, which finishes the proof.
\end{proof}

\subsection{Proofs for coverage times}\label{subsec:proof-coverage-time}
\begin{proof}[Proof of \Cref{thm:coverage-time} (upper bound)]
    A simple rescaling shows that $M(kA,\eps)=M(A,\eps/k)$. Let us fix $\eps,\delta>0$. By the definition of $\beta$, we have for sufficiently large $k$ 
    $$(\eps^{-1}k)^{\alpha+\delta} \geq M(A,\eps/k).$$
    Thus, we can cover $kA$ if we detect all the $M(kA,\eps)$ many centres of the $\eps$-balls using balls of radius $R-\eps$. Let us write
    $$I(\eps) := \E\big[(R-\eps)^{d-\alpha}\I\{R\ge \eps\}\big]$$
    and note that $I(\eps)\to I(0)=\E R^{d-\alpha}$ as $\eps\downarrow0$. 
    By \Cref{lem:detection-time-is-void-probability,lem:volume-asymptotic-random-radii}
        \begin{align*}
            \P\big(T_\mathrm{cov}(kA)>t\big) &\leq \P(\text{some centre is not covered up to time }t \text{ using smaller radii}) \\
            &\leq M(A,\eps/k) \, \P(o \text{ is not covered using radii } R-\eps)\\
            &\leq M(A,\eps/k) \, \exp\big(-\lambda \E|B_{R-\eps}\Gamma_t|\big)
            \leq M(A,\eps/k) \, \exp\big(-\lambda t\Cap{(\alpha,d)}I(\eps)\big).
        \end{align*}
    Thus, for any $t^*\geq0$, we have
    $$\E T_\mathrm{cov}(kA) = \int_0^\infty \P\big(T_\mathrm{cov}(kA)>t\big)\dx t \leq t^* + \int_{t^*}^\infty M(A,\eps/k) \, \exp\big(-\lambda t\Cap{(\alpha,d)}I(\eps)\big)\dx t.$$
    We want to choose $t^*:=t^*(k)$ such that 
    \begin{equation}\label{eq:proof-coverage-time-upper-bound-choice-t}
        1 = (\eps^{-1}k)^{\beta+2\delta} \exp\big(-\lambda t^*\Cap{(\alpha,d)}I(\eps)\big).
    \end{equation}
    In particular, $t^*\uparrow\infty$ as $k\uparrow\infty$ and
    \begin{equation*}
        \begin{aligned}
                \int_{t^*}^\infty& M(A,\eps/k)  \exp\big(-\lambda t\Cap{(\alpha,d)}I(\eps)\big)\dx t
                ={t^*}\int_{1}^\infty M(A,\eps/k) \exp\big(-\lambda (t^* t)\Cap{(\alpha,d)}I(\eps)\big)\dx t \\
                & = \frac{M(A,\eps/k)}{(\eps^{-1}k)^{\beta+2\delta}} t^* \int_{1}^\infty \exp\big(-\lambda t^*\Cap{(\alpha,d)}I(\eps)\big)^{t-1}\dx t
                = \frac{M(A,\eps/k)}{(\eps^{-1}k)^{\beta+2\delta}} \frac{t^*}{\lambda t^*\Cap{(\alpha,d)}I(\eps)}\\
                &\leq \frac{(\eps^{-1}k)^{\beta+\delta}}{(\eps^{-1}k)^{\beta+2\delta}}\frac{1}{\lambda \Cap{(\alpha,d)}I(\eps)}
                = \frac{(\eps^{-1}k)^{-\delta}}{\lambda \Cap{(\alpha,d)}I(\eps)} \xrightarrow{k\uparrow\infty} 0,
        \end{aligned}
    \end{equation*}
    using~\eqref{eq:proof-coverage-time-upper-bound-choice-t} in the second step.
    Thus, we only need to estimate $t^*(k)$. For this, we can bound
    $$t^*(k) \leq \frac{(\beta+2\delta)\log(\eps^{-1}k)}{\lambda \Cap{(\alpha,d)}I(\eps)} = \frac{(\beta+2\delta)\log(k)}{\lambda \Cap{(\alpha,d)}I(\eps)} + c(\eps)$$
    with $c(\eps)$ not depending on $k$. As $\delta>0$ was arbitrarily chosen, we have that
    $$\lim_{k\uparrow\infty} \frac{\E T_\mathrm{cov}(kA)}{\log(k)} \leq \frac{\beta}{\lambda\Cap{(\alpha,d)}I(\eps)}.$$
    Finally, taking $\eps\downarrow0$ proves the claim.
\end{proof}

Regarding the lower bound, we need an estimate for the hitting times of balls that are very far away. In the case of $\alpha=2$, i.e., Brownian motion, this follows from readily available results, see for example~\cite{ItoMcKean1965}. \Cref{lem:hitting-probability} suffices for our purpose.

\begin{proof}[Proof of \Cref{thm:coverage-time} (lower bound)]
We use the alternative characterisation of the Minkowski dimension, where $A\subset\R^d$ has Minkowski dimension $\beta>0$ if
    $$\beta = \lim_{k\uparrow\infty} \log M(kA)/\log(k),$$
    where still $M(kA)$ is the maximal possible number of disjoint balls of radius $1$ with centres in $kA$.
    By this definition of $\beta$, we have for every $\delta>0$ that
    $$k^{\beta-\delta} \leq M(kA) \leq k^{\beta+\delta},$$
    for all sufficiently large $k$. Thus, we know that $\ell A$ is not covered if at least one of those $M(\ell A)$ many points is not detected. Let us denote the random variable
    $$U_t := \#\{\text{undetected points out of the }M(\ell A) \text{ many up to time }t\ge 0\}.$$
    Clearly, $\P(T_\textrm{cov}>t) \geq \P(U_t>0)$. Fix some small $\eps>0$.
    Let us again consider a special time 
    $$t^*:=t^*(k) = \frac{\beta-\delta-\eps}{\lambda\Cap{(\alpha,d)}\E R^{d-\alpha}}\log k.$$
    Then, we see that
    $$\E T_\textrm{cov}(kA) =\int_0^\infty \P\big(T_\textrm{cov}(kA)>t\big) \dx t 
    \geq \int_0^{t^*} \P\big(T_\textrm{cov}(kA)>t\big) \dx t
    \geq t^* \P\big(T_\textrm{cov}(kA)>t^*\big).$$
    As such, the claim follows by showing $\P\big(T_\textrm{cov}(kA)>t^*\big)\to 1$ as $k\uparrow\infty$.
    By the second-moment method
    $$ \P\big(T_\textrm{cov}(kA)>t^*\big) = \P(U_{t^*}>0) \geq \E|U_{t^*}|^2/\E|U_{t^*}^2|,$$
    where, by stationarity,
    $$\E|U_{t^*}| = M(kA) \exp\big(-\lambda \E|B_R(\Gamma_{t^*})|\big).$$
    We are looking for a suitable upper bound for the denominator. Let us denote the $M(kA)$ many centres of the balls by $x_i$. Then, using \Cref{lem:detection-time-is-void-probability} for 2-point sets, we have
    \begin{equation}\label{eq:proof-undetected-points-second-moment}
        \begin{aligned}
            \E|U_{t^*}^2| &= \sum_{i\leq M(kA)} \sum_{j\leq M(kA)} \P(x_i,x_j \text{ both undetected})\\
    &= \E U_{t ^*} + \sum_{i\leq M(kA)}\sum_{i\neq j} \exp\big(-\lambda \E|B_R(\Gamma_{t^*})\cup B_R(\Gamma_{t^*}^{x_j-x_i})|\big)\\
    &= \E U_{t ^*} + \sum_{i\leq M(kA)}\sum_{i\neq j} \exp\big(-2\lambda \E|B_R(\Gamma_{t^*})|\big)\exp\big(\lambda\E|B_R(\Gamma_{t^*})\cap B_R(\Gamma_{t^*}^{x_j-x_i})|\big).
        \end{aligned}
    \end{equation}
    Thus, we have to deal with $M(kA)^2$ many summands and want to show  $\E|B_R(\Gamma_{t^*})\cap B_R(\Gamma_{t^*}^{x_j-x_i})|\to 0$ in some way. First of all, we may safely disregard $o(M(kA)^2)$ many pairs of vertices. Thus, we exclude all $i,j$ which satisfy
    $$|x_i-x_j| <  M(kA)^{1/d} \big/ \log M(kA)=:L(k).$$
    Now, consider for a fixed $r\ge 0$,
    $$F(t,y,r):= \E\big|B_r(\Gamma_{t})\cap B_r(\Gamma_{t}^{y})\big| = \int_{\R^d}\P\big(o\overset{\Gamma_t}\leadsto B_r(x)\text{ and }o\overset{\Gamma_t}\leadsto B_r(x-y)\big) \dx x,$$
    where $o\overset{\Gamma_t}{\leadsto}B_r(X)$ means that $\Gamma_t$ is a path from $o$ into $B_r(x)$, and split the expected overlap into
    \begin{equation}\label{eq:proof-sausage-overlap-summands}
    \begin{split}
        &\E \big|B_R(\Gamma_{t^*(k)})\cap B_R(\Gamma_{t^*(k)}^{y}) \big| \\
        &= \E\big[F(t^*(k),y,R) \I\{R< L(k)^{\alpha/2d}/4\}\big] + \E\big[F(t^*(k),y,R) \I\{R\ge  L(k)^{\alpha/2d}/4\}\big].
        \end{split}
    \end{equation}
    We will show that each of those summands tends to $0$ as $k\uparrow\infty$. Let us start with the second summand. Using \eqref{eq:self-similarity-alpha-stable} as well as \Cref{lem:Boolean-Model-under-shrinkage} and noting that $R\geq L(k)^{\alpha/2d}/4$ and $L(k)^{\alpha/2d}/4\geq t^*(k)^{1/\alpha}$ for large enough $k$, we get
    \begin{equation*}
        \begin{aligned}
         \E\big[F(t^*(k),y,R) \I\{R\ge  L(k)^{\alpha/2d}/4\}\big]
         &\le \E\big[\E|B_R (\Gamma_{t^*(k)})|\I\{R\ge  L(k)^{\alpha/2d}/4\}\big]\\
            &=
            \E\big[t^*(k)^{d/\alpha}\E|B_{Rt^*(k)^{-1/\alpha}}(\Gamma_{1})|\I\{R\ge  L(k)^{\alpha/2d}/4\}\big]\\
            &\leq
            \E\big[R^d\E|B_{1}(\Gamma_{1})|\I\{R\ge  L(k)^{\alpha/2d}/4\}\big]\\
            &= C \E\big[R^d\I\{R\ge  L(k)^{\alpha/2d}/4\}\big]\xrightarrow{k\uparrow\infty} 0,
        \end{aligned}
    \end{equation*}
    since $\E R^d<\infty$.
    Regarding the first summand in \eqref{eq:proof-sausage-overlap-summands}, we may now restrict ourselves to the case of $R \leq {L(k)^{\alpha/2d}/4}$, for which we will derive uniform upper bounds. Using the Markov property, we can bound
        \begin{align*}
            \P\big( o\overset{\Gamma_t}\leadsto B_R(x),\, o\overset{\Gamma_t}\leadsto B_R(x-y)\big)
            &\leq \P \big(\{\exists s\leq s' \leq t\colon X_s \in B_R(x),\,
            X_{s'}\in B_R(x-y)\}\\
            &\qquad\cup \{\exists s' \leq s\leq t\colon X_s\in B_R(x),\, X_{s'}\in B_R(x-y)\}\big)\\
            &\leq \P\big(\exists s\leq s'\leq s+t\colon X_s\in B_R(x),\, X_{s'}\in B_R(x-y)\big)\\
            &\qquad+ \P\big(\exists  s'\leq s\leq s'+t\colon X_s\in B_R(x),\, X_{s'}\in B_R(x-y)\big)\\
        &\leq \P\big(\exists  s \leq t\colon X_s\in B_R(x)\big) \P\big(\exists s' \leq t\colon Y_{s'}\in B_{2R}(-y)\big)\\
        &\qquad+ \P\big(\exists s \leq t\colon Y_{s}\in B_{2R}(y)\big)\P\big(\exists s' \leq t\colon X_{s'}\in B_R(x-y)\big)\\
            &\leq \P\big(\exists s \leq t\colon X_s\in B_{2R}(x)\big) \P\big(\exists  s \leq t\colon X_{s}\in B_{2R}(-y)\big)\\
            &\qquad+ \P\big(\exists  s \leq t\colon X_{s}\in B_{2R}(y)\big)\P\big(\exists s \leq t\colon X_{s}\in B_{2R}(x-y)\big)\\
            &= 2 \P\big(\exists  s \leq t\colon X_s\in B_{2R}(x)\big)\P\big( \exists  s \leq t\colon X_{s}\in B_{2R}(y)\big).
        \end{align*}
   Now, let us consider the ball around the origin of radius $L(k)^{\alpha/2d}$. Then, by the previous inequality and \Cref{prop:exitTime}, we can bound
    \begin{equation*}
        \begin{aligned}
            \int_{B_{L(k)^{\alpha/2d}}(o)} & \P\big( o\overset{\Gamma_{t^*}}\leadsto B_R(x),\, o\overset{\Gamma_{t^*}}\leadsto B_R(x-y)\big) \dx x 
            \leq \int_{L(k)^{\alpha/2d}} 2\P\big( \exists  s \leq t^{*}\colon X_{s}\in B_{2R}(y)\big) \dx x\\
            &\leq |B_{L(k)^{\alpha/2d}}(o)| 2\P\big(\exists  s \leq t^{*}\colon |X_{s}| \geq L(k)-2r\big)\\
            &\leq |B_{L(k)^{\alpha/2d}}(o)| 2\P\big(\exists  s\leq t^{*}\colon |X_{s}| \geq L(k)/2\big)\\
            &\leq C_d L(k)^{\alpha/2} C \frac{t^*(k)}{L(k)^\alpha}
            \leq C \frac{\log(k)^{1/2}}{M(kA)^{\alpha/2d}}\leq C \frac{\log k}{k^{\alpha(\beta-\delta)/2d}} \xrightarrow{k\uparrow\infty}0,
        \end{aligned}
    \end{equation*}
    with constants $C$ changing from line to line but independent of $k$. Thus, we only need to evaluate the integral in $F(t,y,R)$ on the exterior of $B_{L(k)^{\alpha/2d}}(o)$. By \eqref{eq:hitting-probability} in \Cref{lem:hitting-probability} with $L(k)^{\alpha/2d}>t^*(k)^2$ for sufficiently large $k$, we have
    \begin{equation*}
        \begin{aligned}
             \int_{\R^d\backslash B_{L(k)^{\alpha/2d}}(o)}& \P\big( o\overset{\Gamma_{t^*}}\leadsto B_R(x),\, o\overset{\Gamma_{t^*}}\leadsto B_R(x-y)\big) \dx x\\
             &
            \leq \int_{\R^d\backslash B_{L(k)^{\alpha/2d}}(o)} 2\P\big( \exists  s \leq t^{*}\colon X_{s}\in B_{2R}(x)\big) \dx x\\
            &= \int_{L(k)^{\alpha/2d}}^\infty \ell^{d-1} 2\P\big( \exists  s \leq t^{*}\colon X_{s}\in B_{2R}(\ell e_1)\big) \dx \ell\\
            &\leq C \int_{L(k)^{\alpha/2d}}^\infty \ell^{d-1} \frac{t^*(k)^{d+2}} {(\ell-2R)^{d+\alpha-o(1)}} \dx \ell\\
            &\leq C \int_{L(k)^{\alpha/2d}}^\infty \ell^{d-1}
            \frac{t^*(k)^{d+2}} {(\ell/2)^{d+\alpha-o(1)}} \dx \ell
            \\
            &\leq C t^*(k)^{d+2} \int_{L(K)^{\alpha/2d}}^\infty \ell^{-1-\alpha+o(1)} \dx \ell\\
            & = C \frac{(\log k)^{d+2}}{L(k)^{(\alpha-o(1))\alpha/2d}} \leq C \frac{(\log k)^{d+2}}{k^{(\beta-\delta)(\alpha-o(1))\alpha/2d}} \xrightarrow{k\uparrow\infty}0.
        \end{aligned}
    \end{equation*}
    Thus, the first summand in~\eqref{eq:proof-sausage-overlap-summands} can be upper bounded by
    \[\begin{aligned}
\E\big[F(t^*(k),y,R)&\I\{R\le L(k)^{\alpha/2d}/4\}\big] \\
&\leq C \E\Big[\Big[\frac{\log k}{k^{\alpha(\beta-\delta)/2d}} + \frac{(\log k)^{d+2}}{k^{(\beta-\delta)(\alpha-o(1))\alpha/2d}} \Big]\I\{R\le L(k)^{\alpha/2d}/4\}\Big] \xrightarrow{k\uparrow\infty}0,
 \end{aligned}\]
    as the right-hand side does not depend on $R$. As such, we have proven that the terms in~\eqref{eq:proof-sausage-overlap-summands} are $o(1)$. Plugging this estimate into~\eqref{eq:proof-undetected-points-second-moment}, we get 
    \begin{equation*}
        \begin{aligned}
            \E|U_{t^*}^2| &= \E U_{t ^*} + \sum_{i\leq M(kA)}\sum_{i\neq j} \exp\big(-2\lambda \E|B_R(\Gamma_{t^*})|\big)\exp\big(\lambda\E|B_R(\Gamma_{t^*})\cap B_R(\Gamma_{t^*}^{x_j-x_i})|\big)\\
            &\leq \E U_{t ^*} + \sum_{i\leq M(kA)}\sum_{i\neq j} \exp\big(-2\lambda \E|B_R(\Gamma_{t^*})|\big)\exp\big(\lambda \E[F(t^*,x_j-x_i,R)]\big)\\
            &\leq \E U_{t ^*} + \sum_{i\leq M(kA)}\sum_{i\neq j} \exp\big(-2\lambda \E|B_R(\Gamma_{t^*})|\big)(1+o(1))\\
            &= \E U_{t ^*} + M(kA)^2 \exp(-2\lambda \E|B_R(\Gamma_{t^*})|) (1+o(1))\\
            &= \E U_{t ^*} + (\E U_{t ^*})^2 (1+o(1)),
        \end{aligned}
    \end{equation*}
    as $k\to\infty$. The last step we need to confirm is that $\E U_{t ^*}\to\infty$ as $k\uparrow\infty$ to prove that
    $$\frac{\E|U_{t^*}|^2}{\E|U_{t^*}^2|} \geq \frac{\E|U_{t^*}|^2}{\E|U_{t ^*}|+\E|U_{t^*}|^2(1+o(1))} \to 1.$$
    But, indeed,
    \begin{equation*}
        \begin{aligned}
            \E U_{t ^*} &= M(kA)\exp\big(-\lambda t^* \Cap{(\alpha,d)} (1+o(1)\big) \\
            &= M(kA) \exp\Big(-\lambda \Cap{(\alpha,d)} \E\Big[R^{d-\alpha}\frac{(\beta-\delta-\eps) \log k}{\lambda \Cap{(\alpha,d)} \E R^{d-\alpha}}\Big] \Big)\\
            &\geq k^{\beta-\delta} k^{-\beta+\delta+\eps} = k^\eps \xrightarrow{k\uparrow\infty}\infty.
        \end{aligned}
    \end{equation*}
    Putting everything together indeed yields that
    $$\E T_\textrm{cov}(kA) \geq t^* (1+o(1)) = \frac{\beta-\delta-\eps}{\lambda\Cap{(\alpha,d)} \E R^{d-\alpha} }\log k (1+o(1)),$$
    for every $\delta,\eps>0$, which concludes the proof.
 \end{proof}

\subsection{Proof of percolation time}\label{sec:proof_perc_time}
Recall that we fixed a dimension \(d\ge 2\), an index \(\alpha\in(0,2)\) and a Poisson intensity \(\lambda>\lambda_c\) so that \(\G_t\) contains an unbounded component at each point in time \(t\geq 0\) with probability one. We prepare the general strategy outlined in \Cref{sec:outlook}, and observe that nodes not only move but have heterogeneous radii attached that do not change with time. As such, we sort the nodes of the Poisson point process into finitely many categories depending on the sizes of their radii. These categories will decompose the Poisson point process into finitely many independent Poisson point processes, so we require a decoupling result that applies to all of these processes simultaneously. To this end, let \(M\in\N\) and \(\eta_0^1,\dots,\eta_0^M\) be independent Poisson processes of intensity \(\lambda/M>0\) so that \(\eta_0=\eta_1+\dots+\eta_M\) is a Poisson process of intensity \(\lambda\), where each vertex is independently assigned one of \(M\) marks with probability \(1/M\). We equip each Poisson point with an independent standard isotropic $\alpha$-stable process with \(\alpha\in(0,2)\) and denote by \(\eta^j_t\) the location of the nodes of \(\eta_0^j\) at time \(t\).

\begin{definition}[Good configurations]\label{def:GoodBox}
    Let \(V\in\R\) and \(\xi\in(0,1)\) and consider \(\eta_0=\eta_1+\dots+\eta_M\), an independently marked Poisson process of intensity \(\lambda>0\) with \(M\) marks, where each mark is assigned with probability \(1/M\). 
    We say a vertex configuration in a box $Q_V$ of volume $V$ is \(\xi\)--\emph{good} for \(\eta_0\), if, for all \(j=1,\dots M\), the locations of \(\eta^j_t\cap Q_V\) contain as a subset a Poisson point process of intensity \((1-\xi)\lambda/M\) that is independent of everything else.      
\end{definition}

The next result provides the key argument for the aforementioned decoupling. 

\begin{prop}[Decoupling]\label{prop:decoupling}
    Let \(\alpha\in(0,2)\), \(\varepsilon,\xi\in(0,1)\), and let \(Q=Q_V\) be a volume-\(V\) box. Define for \(i\in\N\) the event
    \begin{equation}\label{eq:goodnessEvent}
        \mathcal{A}_i := \{\text{the box }Q_V\text{ is }\xi\text{--good at time }i\}.
    \end{equation}
    Then, if $V=\omega(\log^{d/\alpha}(t))$, one can pick a function $f(t)=\omega(1)$ depending only on \(\varepsilon,\xi\) such that, for \(t\in\N\) large, we have
    \begin{equation}\label{eq:goodBoxBound}
        \P\Big(\sum_{i=0}^{t-1}\1_{\mathcal{A}_i}>(1-\varepsilon)t\Big)\geq 1-\Big(V^{1/d}+ c\big(\tfrac{t}{\log t}\big)^{1/\alpha}\Big)^d\exp(- t/f(t)).
    \end{equation}
\end{prop}

We give the technical proof in the next \Cref{sec:proofDecoupling} and first explain how it can be used to prove our third main result.

\begin{proof}[Proof of \Cref{thm:percTime}.]
 We first consider the percolation time for an immobile node at the origin, i.e.\ \(g\equiv o\) and 
 \[
    T_\perc= T_\perc^o = \inf\{t\geq 0\colon \exists X^i_t \in \Psi_t^\infty \text{ such that }|X^i_t|<R_i\}.
\]
We comment on the more general movements \(g\) at the end of the proof. 

The main step of our proof is to first localise the problem by restricting the model to a large box and an application of the decoupling result \Cref{prop:decoupling} afterwards that allows us to independently resample the model within that box. We start by collecting some preliminaries that bring us in the position to do so.
    
    \paragraph{Step 1: Construction and discretisation.}
        We start by constructing the model on the points of a uniformly-marked Poisson process. To this end, assign each vertex of \(\Psi_0=\{X_0^i\}_{i\in\N}\) an independent mark distributed uniformly on \((0,1)\) that we denote by \(\{U_i\}_{i\in\N}\). Let \(\bar f_R(r)=\P(R>r)\) be the tail distribution function of the radius distribution and \(\bar f_R^{-1}\) its generalised inverse. Then, \(R_i\) has the same distribution as \(\bar f_R^{-1}(U_i)\) and the model can thus be constructed using the uniformly marked Poisson process. Put differently, at each time \(t\), we can write
        \[
            \G_t = \bigcup_{X_t^i} B_{\bar f_R^{-1}(U_i)}(X_t^i).
        \]

        Now, fix \(\lambda>\lambda_c\) to guarantee the existence of an unbounded connected component. We next define a supercritical instance of this model whose radii take only \(M\) different values, each with equal probability, and that thus reduces to an independently-marked Poisson process of intensity \(\lambda\) with \(M\) marks. 

        To this end, let \(\delta\in(0,1)\), \(N\in\N\) large and define the partition 
        \[
            \cP_{N,\delta}=\{\bar{f}_R(N),\bar{f}_R(N)+\delta,\dots,1-\delta,1\}
        \]
        of the mark space $(0,1)$ where we assume without loss of generality that this is possible. We define the truncated Boolean model with radius distribution restricted to $\cP_{N,\delta}$ as follows. We start with the regular Boolean model with radius distribution $\bar f_R$ and remove all vertices with radius larger than $N$ (that is, all marks smaller than $\bar{f}_R(N)$). Next, we discretise the remaining vertex radii by ``rounding down'' the radii with respect to the partition $\cP_{N,\delta}$. More precisely, if $R=\bar{f}_R^{-1}(U)$ is the radius of a given vertex and $U\in (n\delta,(n+1)\delta]$, we set $R$ to \(\bar f_R^{-1}((n+1)\delta)\). We denote the corresponding Boolean model by \(\G_0^{N,\delta}\). It is clear that \(\G_0^{N,\delta}\subset \G_0\) and the same remains true at all points in time \(t\) since the nodes in both models follow the same trajectories (in fact both models only differ in the associated radius distribution). Another immediate consequence of this construction is that if the discretised model \(\G_0^{N,\delta}\) is supercritical, the original model \(\G_0\) is supercritical as well. In fact, when looking at the region occupied by the discretised model, it is strictly contained in the region occupied by the original model. Furthermore, this monotonicity trivially extends to the infinite component of the discretised model; if such a component exists, it is by necessity fully contained in the infinite component of the original model.
        
        Finally, observe that \(\G_0^{N,\delta}\) is constructed from a Poisson process where each vertex gets independently assigned one of 
        \[
            M=\max\{n\colon \bar f_R(N)+n\delta\leq 1\}
        \] 
        marks, each getting assigned with probability \(1/M\). In the following, we may always use the superscripts \(N,\delta\) to indicate that the quantity in question refers to the truncated model. 

        \paragraph{Step 2: The discretised model is supercritical.}
        Let us argue that the discretised model is supercritical if \(N\) is chosen large and \(\delta\) is chosen small enough. That is, \(\lambda>\lambda_c^{N,\delta}\). Indeed, we first observe that, by the truncation result~\cite[Corollary~1.6]{dembin2022sharpess}, for each \(\lambda>\lambda_c\), there exists \(N\) sufficiently large such that the truncated Boolean model, where all radii larger than \(N\) are removed, is supercritical still. Since this truncated model is of finite range, we may afterwards apply the approximation result \cite[Theorem~3.7]{Meester1996} to infer that for \(\delta\) sufficiently small, we still have \(\lambda>\lambda_c^{N,\delta}\), as claimed. 
    
    \paragraph{Step 3: Applying the decoupling result.}
    We now work exclusively in the supercritical discretised model, i.e., \(N\) and \(\delta\) are chosen appropriately. Note here that monotonicity of the unbounded component implies
    \(
        T_{\perc}\leq T_{\perc}^{N,\delta}.
    \)
    Now, pick a box volume \(V=V_t\geq t^{d/(d-1)}\geq N^d\) and consider the box \(Q_V\). Note that no ball can cover the whole box by the truncation in place. Fix a \(\xi>0\) so small that \((1-\xi)\lambda>\lambda_c^{N,\delta}\), fix some \(\varepsilon\in(0,1)\), and denote, for \(t\in\N\), by \(\Ecal_t\) the event that
     \[
     	\sum_{i=0}^{t-1}\1_{\mathcal{A}_i}>(1-\varepsilon)t, 
     \] 
     where \(\mathcal A_i\) is the event that the box is good at time \(i\), cf.~\eqref{eq:goodnessEvent}. Furthermore, we denote by \(\C^{N,\delta}_s(Q_V)\) the largest connected component of \(\G^{N,\delta}_s\cap Q_V\), and define for each \(j\in\N\) the event 
	\[
		\Hcal_j := \big\{|\C^{N,\delta}_j(Q_V)|\geq \rho V \text{ and } \C^{N,\delta}_j(Q_V)\subset \G^{(N,\delta), \infty}_j\big\}.
	\]
	 Put differently, \(\Hcal_j\) is the event that, at time \(j\), the box \(Q_V\) contains a component of linear proportion \(\rho\) that intersects with the unbounded component of the discretised model. Indeed, \(\C^{N,\delta}_j(Q_V)\) is the \emph{unique} giant component of asymptotic size \(\theta^{N,\delta}(\lambda):=\P(o\leftrightarrow\infty \text{ in }\G_0^{N,\delta})\), whose existence and uniqueness are implied by super-criticality and the truncation of radii that makes the results of~\cite{Penrose1996} applicable. Choosing \(\rho=\theta^{N,\delta}((1-\xi)\lambda)/2\), we thus infer from~\cite[Theorem~1]{Penrose1996}, that
	\begin{equation*}
        \P(\Hcal_j^c)\leq \exp(-c V^{(d-1)/d})\leq \exp(-c t),	    
	\end{equation*}
	noting that \(0<\theta^{N,\delta}((1-\xi)\lambda)\leq \theta^{N,\delta}(\lambda)\) by our choice of \(\xi\), and that the constant in the exponential bound depends on \(\lambda\) and \(\xi\). Combining these observations yields for any \(t\in\N\) large enough,
	\[
		\begin{aligned}
			\P_o(T_\perc\geq t)
			&
				\leq \P_o\Big(\Ecal_t\cap \bigcap_{j=0}^{t-1}\big(\{o\not\in\C^{N,\delta}_j(Q_V)\}\cap\{|\C^{N,\delta}_j(Q_V)|\geq \rho V\}\big)\Big) + \exp(-c t/f(t)),
		\end{aligned}
	\]
	by \Cref{prop:decoupling}, where the constant \(c\) is chosen sufficiently small according to our choices for \(\xi\), \(\varepsilon\), and \(f\). To bound the remaining probability, observe that, on \(\Ecal_t\), there exist \(k=\lfloor(1-\varepsilon)t\rfloor\) many time steps \(\tau_1,\dots,\tau_k\), where there exist independent Poisson processes \(\eta^1,\dots,\eta^k\) in \(Q_V\) independently marked with the \(M\) marks of equal probability, of intensity \((1-\xi)\lambda>\lambda_c^{N,\delta}\), such that \(\eta^i\subset \Psi^{N,\delta}_{\tau_i}\cap Q_V\). Moreover, \(\eta^i_j\), the restriction of \(\eta^i\) to those vertices with mark \(j\) form an independent Poisson process of intensity \((1-\xi)\lambda/M\), which is a subset of those vertices of the original process at time \(\tau_i\) also restricted to mark \(j\). Let us write \(\Ccal(\eta^i)\) for the largest connected component constructed inside \(Q_V\) on the points of \(\eta^i\). Clearly, \(\Ccal(\eta^i)\subset \C^{N,\delta}_{\tau_i}(Q_V)\). Moreover, as \((1-\xi)\lambda>\lambda_c(N,\delta)\), with high probability, \(\Ccal(\eta^i)\) is the giant component of \(\eta^i\) of size \(\theta_{N,\delta}((1-\xi)\lambda)\), implying \(|\Ccal(\eta^i)|\geq \rho V\) with probability approaching one, as \(t\to\infty\), by choice of \(\rho\). Using the independence of all \(\eta^i\), we thus infer for \(t\) large enough
    \begin{equation*}
		\begin{aligned}
			\P_o\Big(\Ecal_t\cap &\bigcap_{j=0}^{t-1}\big(\{o\not\in\C_{N,\delta}^V(j)\}\cap\{|\C_{N,\delta}^V(j)|\geq \rho V\}\big)\Big)\\
			&
				\leq \Big(\P_o\big(o\not\in \Ccal(\eta^1), |\Ccal(\eta^1)|\geq \rho V\big)+\P(\Ccal(\eta^1)<\rho V)\Big)^k
			\\ &
				\leq \big((1-\rho)+o(1)\big)^k \leq \exp(-c t),
		\end{aligned}
    \end{equation*}
	for some \(c\) that, again, only depends on \(\lambda\) and \(\xi\), proving the statement for an immobile origin.

    To generalise the proof to any \(g\) satisfying the assumption of the theorem, note that we only require that the target node is in $Q_V$ at the observation times $\{\tau_1,\dots,\tau_k\}$, as its precise location does not affect the argument. This is ensured if $g[0,t]\subset Q_V$. However, in order to obtain the desired tail bound, we require (cf.\ \Cref{prop:decoupling}) that $V^{1/d}$ is of order at most $\exp(o(t/f(t)))$.
\end{proof}

\subsection{Proof of the decoupling} \label{sec:proofDecoupling}
In order to prove the crucial decoupling result \Cref{prop:decoupling}, let us collect two preliminaries on soft local times and implied mixing properties. 

For this, we start by stating a result from~\cite{Popov2015}, which will let us couple the locations of our particle system after they have moved with an independent Poisson point process on a locally compact and Polish metric space $\mathbb{G}$.

\begin{prop}[Soft local times]\label{prop-softLocalTimes}
	Let $J$ be an at most countable index set and let \(\{Z_j\}_{j\in J}\) be a collection of points distributed independently on $ \mathbb{G}^d $ according to a family of probabilities, given by \(g_j\colon\mathbb{G}^d\rightarrow\mathbb{R}\), \(j\in J\). Define for all \(y\in \mathbb{G}^d\) the soft local-time function \(H_J(y):=\sum_{j\in J}\xi_jg_j(y)\), where the \(\xi_j\) are i.i.d.\ exponential random variables of mean \(1\). Let \(\psi\) be a Poisson point process on $ \mathbb{G}^d $ with intensity measure \(\rho\colon\mathbb{G}^d\rightarrow\mathbb{R}\) and define the event
	\(
	E:=\{\psi\subseteq\{Z_j\}_{j\in J}\},
	\)
    i.e., the particles belonging to $\psi$ are a subset of $\{Z_j\}_{j\in J}$.
	Then, there exists a coupling $\mathbb{Q}$ of \(\{Z_j\}_{j\in J}\) and \(\psi\), such that
	\[
		\mathbb{Q}(E)\geq\mathbb{Q}\big(H_J(y)\geq\rho(y),\;\forall y\in \mathbb{G}^d\big).
	\]
\end{prop}
\begin{proof}
	The coupling is introduced in~\cite[Section~4]{Popov2015} and proven in~\cite[Corollary~4.4]{Popov2015}. A reformulation of the construction for particles on a graph can be found in~\cite[Appendix~A]{Hilario2015}, and our claim corresponds to~\cite[Corollary~A.3]{Hilario2015}.
\end{proof}

\begin{theorem}[Mixing]\label{thrm:mixing}
	There exist constants \(c_0\), \(c_1\), \(C>0\) such that the following holds.
	Fix large enough \(K>\ell>0\), \(\epsilon\in(0,1)\). Consider the cube \(Q_K\subset\R^d\) tessellated into subcubes \((T_i)_{i}\subset Q_K\) of side length \(\ell\).  
	Let \((x_j)_{j}\subset  Q_{K}\) be the locations at time \(0\) of a collection of particles, such that each subcube \( T_i\) contains at least \(\beta|T_i|\) particles for some \(\beta>0\).
	Assume that $\ell$ is sufficiently large so that \(\lfloor\beta|T_i|\rfloor>0\) for all subcubes \(T_i\).
	Let \(\Delta\geq c_0\ell^\alpha\epsilon^{-4/\kappa}\) where \(\kappa\) is as in \Cref{prop:hoelder}.
	For $J$ being the index set of all particles in $Q_K$ and $j\in J$, denote by \(Y_j\) the location of the \(j\)-th particle at time \(\Delta\). 
	Fix \(K'>0\) such that \(K-K'\geq c_1{\Delta}^{1/\alpha}\epsilon^{-1/d}\). Then, there exists a coupling \(\mathbb{Q}\) of an independent Poisson point process \(\psi\) with intensity measure \(\zeta(y)=\beta(1-\epsilon)\), \(y\in \R^d\), and \((Y_j)_{j\in J}\) such that within \( Q_{K'}\subset  Q_K\), \(\psi\) is a subset of \((Y_j)_{j\in J}\) with probability at least
	\[
		1-|Q_{K'}|\exp\big(-C\beta\epsilon^2\Delta^{d/\alpha}\big).
	\]
\end{theorem}
\begin{proof}
    Using \Cref{prop-softLocalTimes}, there exists a coupling $\mathbb{Q}$ of an independent Poisson point process $\psi$ with intensity measure $\beta(1-\epsilon)$ and the locations of the vertices $(Y_j)_{j\in J}$ which are distributed according to the density functions $q_{\Delta}(x_j,y)$, $y\in \R^d$, such that the vertices of $\psi$ are a subset of $(Y_j)_{j\in J}$ with probability at least
    \[
        \mathbb{Q}(H_J(y)\geq \beta(1-\epsilon),\forall y\in Q_{K'}).
    \]
    Here, $H_j(y)=\sum_{j\in J}\xi_j q_{\Delta}(x_j,y)$ and $(\xi_j)_{j\in J}$ are i.i.d.\ exponential random variables with parameter $1$. Considering the converse event and applying Markov's inequality for an arbitrary $\kappa>0$, we get
    \begin{align*}
        \mathbb{Q}\big(\exists y\in Q_{K'}\colon H_J(y)<\beta(1-\epsilon)\big)&\leq \int_{Q_{K'}}\mathbb{Q}\big(H_J(y)<\beta(1-\epsilon)\big)\dx y\leq \int_{Q_{K'}}{\rm e}^{\kappa\beta(1-\epsilon)}\E_{\mathbb{Q}}{\rm e}^{-\kappa H_J(y)}\dx y.
    \end{align*}
    Let $c_1$ be a large positive constant that we will fix later and consider the distance 
    \[
        L:=c_1\Delta^{1/\alpha}\epsilon^{-1/d}.
    \]
    Next, let $J'$ be any subset of $J$ that satisfies the following: for each $T_i$, $J'$ contains exactly $\lceil\beta|T_i|\rceil$ vertices inside $T_i$. If such a subset is not unique, pick one according to an arbitrary ordering of the indices in $J$. Furthermore, define $J'(y)\subseteq J'$ to be the set of $j\in J'$ that satisfy $|x_j-y|\leq L$. Finally, define $H'(y)$ in the same way as $H_J(y)$, but by restricting the sum to $j\in J'(y)$. It follows immediately that $H_J(y)\geq H'(y)$ and so
    \begin{align}
        \E_{\mathbb{Q}}\exp(-\kappa H'(y))&=\prod_{j\in J'(y)}\E_{\mathbb{Q}}\exp(-\kappa q_{\Delta}(x_j,y)\xi_j)=\prod_{j\in J'(y)}(1+\kappa q_{\Delta}(x_j,y))^{-1}\label{eq:taylorApplication}.
    \end{align}
    Using \Cref{thrm:heatkernel} we have that $q_{\Delta}(x,y)\leq C(\Delta^{-d/\alpha}\wedge c_2\Delta|x-y|^{-d-\alpha})$ for some constants $C$ and $c_2$, for all $y\in Q_{K'}$ and all $x\in\bigcup T_i$, where the union runs across all $T_i$ that contain at least one $x_j$ with $j\in J'(y)$. If we set $\kappa= \tilde C\epsilon\Delta^{d/\alpha}$ for $\tilde C := ((1\wedge c_2)4C)^{-1}$, then
    \[
        \sup_{x\in B_{L+\sqrt{d}\ell}(y)}\kappa q_{\Delta}(x,y)\leq \kappa \sup_{x\in B_{L+\sqrt{d}\ell}(y)}C(\Delta^{-d/\alpha}\wedge c_2\tfrac{\Delta}{|x-y|^{d+\alpha}})\leq \kappa C(1\wedge c_2)\Delta^{-d/\alpha}<\epsilon/4.
    \]
    
    Applying the inequality $1+x\geq \exp(x-x^2)$, $|x|\leq1/2$ to \eqref{eq:taylorApplication} and then using the above supremum bound, we calculate
    \begin{align*}
        \prod_{j\in J'(y)}(1+\kappa q_{\Delta}(x_j,y))^{-1}&\leq \prod_{j\in J'(y)}\exp\big(-\kappa q_{\Delta}(x_j,y)(1-\kappa q_{\Delta}(x_j,y))\big)\\
        &\leq \exp\Big(-\sum_{j\in J'(y)}\kappa q_{\Delta}(x_j,y)\big(1-\sup_{x\in B_{L+\sqrt{d}\ell}(y)}\kappa q_{\Delta}(x,y)\big)\Big)\\
        &\leq\exp\Big(-\kappa\sum_{j\in J'(y)} q_{\Delta}(x_j,y)(1-\epsilon/4)\Big).
    \end{align*}
    Once we establish $\sum_{j\in J'(y)} q_{\Delta}(x_j,y)\geq \beta(1-\epsilon/2)$ and use the definition of $\kappa$, the proof will be complete. To that end, define the following association: for each $T_i$ and each $x_j\in T_i$, set $x_j'=\argmax_{w\in T_i}q_{\Delta}(w,y)$. In other words, associate to each particle $x_j$ the location in the subcube $T_i$ containing $x_j$ that maximises $q_{\Delta}(\cdot,y)$. If this choice is not unique, we can choose this location arbitrarily; for example, by choosing the location with the smallest first coordinate. If that choice is still not unique, pick from the remaining candidates the one with the smallest second coordinate, etc. With this choice, a simple application of the triangle inequality gives
\[
    \sum_{j\in J'(y)}q_{\Delta}(x_j,y)\geq \sum_{j\in J'(y)}\big(q_{\Delta}(x_j',y)-|q_{\Delta}(x_j',y)-q_{\Delta}(x_j,y)|\big).
\]
For each $T_i$ that is entirely contained in $B_R(y)$, we have
\begin{align*}
    \sum_{\substack{j\in J'(y)\\x_j\in T_i}}q_{\Delta}(x_j',y)&=\max_{w\in T_i}q_{\Delta}\sum_{\substack{j\in J'(y)\\x_j\in T_i}}1\geq \max_{w\in T_i}q_{\Delta} \beta|T_i|\geq \int_{T_i}\beta q_{\Delta}(z,y)\dx z,
\end{align*}
where the second inequality follows from our definition of the set $J'$ and the fact that if $J'(y)$ contains a single vertex from $T_i$, it must contain all of them. To take advantage of this fact, we consider the set of all locations $z\in \R^d$ such that $|z-y|\leq L-\sqrt{d}\ell$, i.e., $B_{L-\sqrt{d}\ell}(y)$. 
This radius is positive since by definition, $L$ is proportional to $\ell$ and $c_1$ is assumed to be large. By the above discussion, if $z\in B_{L-\sqrt{d}\ell}(y)$, then if a vertex $x_j$ has $x_j'=z$ and $j\in J'$, then necessarily $j\in J'(y)$. Putting everything together, we have
\begin{align*}
    \sum_{j\in J'(y)}q_{\Delta}(x_j',y)&\geq \int_{B_{L-\sqrt{d}\ell}(y)}\beta q_{\Delta}(z,y)\dx z=\beta\int_{B_{L-\sqrt{d}\ell}(y)}q_{\Delta}(y,z)\dx z.
\end{align*}
By \Cref{prop:exitTime},
\begin{align*}
    \beta\int_{B_{L-\sqrt{d}\ell}(y)}q_{\Delta}(y,z)\dx z&\geq \beta\mathbb{P}_y(\tau_{L-\sqrt{d}\ell}\geq \Delta)\geq \beta\big(1-c_4\tfrac{\Delta}{(L-\sqrt{d}\ell)^{\alpha}}\big)\\
    &\geq \beta\big(1-c_4\tfrac{\Delta}{(c_1\Delta^{1/\alpha}\epsilon^{-1/d}-\sqrt{d}\ell)^{\alpha}}\big)\geq \beta(1-\epsilon/4),
\end{align*}
where we have set $c_1$ large enough with respect to $c_4$ for the last inequality to hold.

It remains to bound the term $\sum_{j\in J'(y)}|q_{\Delta}(x_j',y)-q_{\Delta}(x_j,y)|$. Let $I$ be the index set of all subcubes that contain a vertex $x_j$ from the set $(x_j)_{j\in J'(y)}$. Then, applying  \Cref{prop:hoelder}, we have
\begin{align*}
    \sum_{j\in J'(y)}|q_{\Delta}(x_j',y)-q_{\Delta}(x_j,y)|&=\sum_{i\in I}\sum_{\substack{j\in J'(y)\\x_j\in T_i}}|q_{\Delta}(x_j',y)-q_{\Delta}(x_j,y)|\\
    &\leq\sum_{i\in I}\sum_{\substack{j\in J'(y)\\x_j\in T_i}}c\Delta^{-(d+\kappa)/\alpha}(\sqrt{d}\ell)^\kappa\leq\sum_{i\in I}2\beta|T_i| \tilde c \ell^\kappa\Delta^{-(d+\kappa)/\alpha}.
\end{align*}
Finally, we note that $\sum_{i\in I}|T_i|$ is smaller than $\hat c L^d$ for some constant $\hat c$ and so the last expression can be bounded from above by
\(
    2\beta \hat c L^d\ell^\kappa\Delta^{-(d+\kappa)/\alpha}.
\)
Using the definition of $L$ and setting $c_0$ large enough with respect to $\hat c$, this can finally be bounded from above by \(\beta\epsilon/4\).
Putting everything together, this concludes the proof.
\end{proof}

\begin{remark}
    We note that, when $\alpha=2$, we recover the result from \cite{Peres2011} for Brownian motion.
\end{remark}

We now come to the proof of the decoupling result. We will follow a strategy similar to the ones in~\cite{Peres2011} and~\cite{Gracar2022b}. Crucially, although we will consider the vertices with different marks separately, we will make the argument for all of them simultaneously, as we require that all of the Poisson point processes in the superposition are ``well behaved'' for the same $(1-\epsilon)t$ many time steps.
Before we continue, we remind the reader that each of these $M$ point processes has a uniform density of $\lambda/M$ as part of the superposition of combined intensity $\lambda$.  Therefore, although we carefully track the many marked point processes to apply appropriate union bounds, we note that due to their shared intensity, the locations of the vertices behave qualitatively in the same way, regardless of their associated marks.

\begin{proof}[Proof of \Cref{prop:decoupling}]
    Define $\kappa:=(\alpha/\log 2-f_\kappa(t))\log t=O(\log t)$ to be the number of scales we will consider in the upcoming multi-scale argument, where $f_\kappa(t)\in(0,a/\log 2)$ is a (not necessarily strictly) monotonically decreasing function that we will fix later. We begin the argument in a box of volume 
    %$V_1:=V^\alpha$ and conclude at the desired volume $V_\kappa= V$, 
    $V_1>V$ (which we will fix momentarily) and conclude at the desired volume $V_\kappa= V$,
    decreasing the volume at each step of the argument, so that $V_1>V_2>\cdots>V_\kappa$. At each step of the argument, we tessellate the box of volume $V_i$ into sub-boxes of volume $v_i\ll V_i$, where these sub-boxes are similarly decreasing in volume, that is, $v_1>v_2>\cdots>v_\kappa$; see \Cref{fig:spatialtessellation}. For the time being, we impose no assumptions on how $v_i$ should scale with $t$ or $V_i$. In order for our construction to work, we set $V_1:=(V^{1/d}+B\sum_{i=1}^{\kappa-1}v_i^{1/d})^d$ where $B$ is a constant to be fixed later. This choice gives us enough spatial slack to be able to repeatedly apply \Cref{thrm:mixing} and obtain the result for the entire box $Q_V$.

    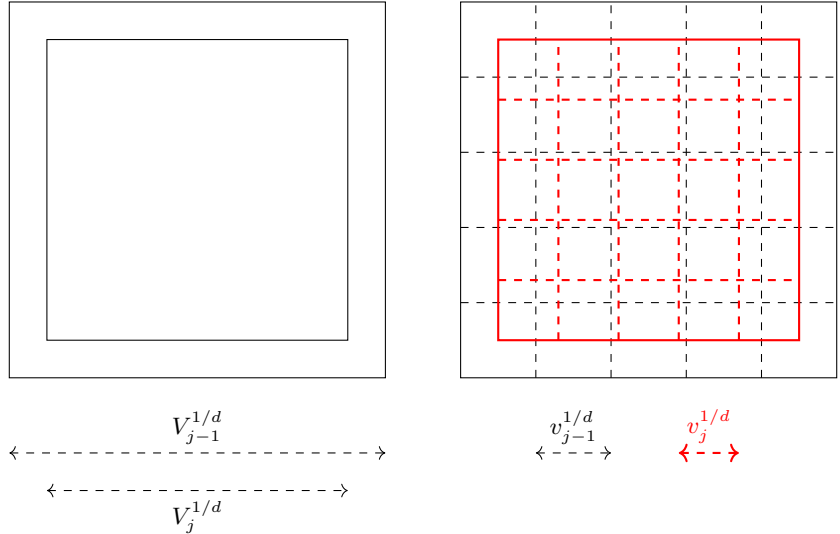
\begin{figure}[!ht]
	\begin{center}
    \begin{tikzpicture}
	\tikzstyle{every node}=[font=\small]
	\draw (0,0) rectangle (5,5);
	\draw[<->,dashed] (0,-1) -- (2.5,-1) node[above] {$V_{j-1}^{1/d}$} -- (5,-1);
	
	\draw (0.5,0.5) rectangle (4.5,4.5);
	\draw[<->,dashed] (0.5,-1.5) -- (2.5,-1.5) node[below] {$V_j^{1/d}$} -- (4.5,-1.5);
	
	\draw (6,0) rectangle (11,5);
	\draw[dashed] (6,1) -- (11,1);
	\draw[dashed] (6,2) -- (11,2);
	\draw[dashed] (6,3) -- (11,3);
	\draw[dashed] (6,4) -- (11,4);
	
	\draw[dashed] (7,0) -- (7,5);
	\draw[dashed] (8,0) -- (8,5);
	\draw[dashed] (9,0) -- (9,5);
	\draw[dashed] (10,0) -- (10,5);
	
	\draw[<->,dashed] (7,-1) -- (7.5,-1) node[above] {$v_{j-1}^{1/d}$} -- (8,-1);
	
	\draw[red,thick] (6.5,0.5) rectangle (10.5,4.5);
	\draw[dashed,red,thick] (6.5,0.5+4/5) -- (10.5,0.5+4/5);
	\draw[dashed,red,thick] (6.5,0.5+8/5) -- (10.5,0.5+8/5);
	\draw[dashed,red,thick] (6.5,0.5+12/5) -- (10.5,0.5+12/5);
	\draw[dashed,red,thick] (6.5,0.5+16/5) -- (10.5,0.5+16/5);
	
	\draw[dashed,red,thick] (6.5+4/5,0.5) -- (6.5+4/5,4.5);
	\draw[dashed,red,thick] (6.5+8/5,0.5) -- (6.5+8/5,4.5);
	\draw[dashed,red,thick] (6.5+12/5,0.5) -- (6.5+12/5,4.5);
	\draw[dashed,red,thick] (6.5+16/5,0.5) -- (6.5+16/5,4.5);
	
	\draw[<->,dashed,thick,red] (6.5+12/5,-1) -- (6.5+28/10,-1) node[above] {$v_{j}^{1/d}$} -- (6.5+16/5,-1);

    \end{tikzpicture}
    \caption{The spatial multi-scale recursion. Note that the difference $V_{j-1}-V_j$ is proportional to the value $v_{j-1}$ by \eqref{eq:v_jm_j} and \eqref{eq:v_geometric}, allowing us to apply \Cref{thrm:mixing}.}\label{fig:spatialtessellation}
    \end{center}
    \end{figure}

    Let $(\xi_j)_{j\in\{1\dots,\kappa\}}$ be a sequence satisfying
    \[
        \xi/2=\xi_1<\xi_2<\cdots<\xi_\kappa=\xi
    \]
    and
    \[
        \xi_j-\xi_{j-1}=\xi/(2(\kappa-1)),\quad\forall j\in\{2,\dots,\kappa\}.
    \]
    We say a sub-box is \emph{good} at a given time step (that is, at a given time $t\in \N$) for scale $j$, if it contains at least $(1-\xi_j)\lambda v_j/M$ vertices for each of the $M$ marks of the point process.

    \paragraph{Step 1: First scale.} We begin at the largest scale $1$ and define $D_1$ to be the event that all sub-boxes of volume $v_1$ inside the box of volume $V_1$ are good for all time steps (that is, integer times) in $[0,t]$. By the stationarity of the Poisson point process of vertices of a given mark, their number inside a fixed sub-box at a fixed time step is a Poisson random variable with mean $\lambda v_1/M$. Using a standard Chernoff bound, we have that with probability larger than $1-\exp(\xi_1^2\lambda v_1/(2M))$, this number is larger than $(1-\xi_1)\lambda v_1/M$. The total number of sub-boxes of volume $v_1$ in $Q_{V_1}$ is $O(V_1)$ by our assumptions on $v_1$. Taking a union bound across all sub-boxes of scale $1$, all time steps in $[0,t]$ and the $M$ marks, we get
    \begin{equation}\label{eq:D_1}
        \P(D_1)\geq 1-MV_1 t\exp\big(-\xi_1^2\tfrac{\lambda}{M} v_1/8\big)\geq 1-V^\alpha\exp\big(-c_1v_1 +\log(t)\big).
    \end{equation}
    
    As we want the probability of $D_1$ to go to $1$ with $t$, this imposes the relationship that $v_1=\omega(\log(t))$.

    \paragraph{Step 2: Smaller scales.} 
    When moving from a larger scale to a smaller one, we will discard some of the time steps and only retain a large proportion for which the vertex behaviour agrees with our requirements. Let $s_j$ be the number of time steps we consider at scale $j$, starting with $s_1=t$, meaning that, as seen above, all of the $t$ time steps were considered. When going from scale $j-1$ to $j$, we will divide the steps into $4$ groups of $m_{j}$ consecutive time steps, starting with $m_1=t$. Put differently, we start with a single interval of $t$ time steps at scale $1$ in agreement with how we treated scale 1 above, which we then subdivide into 4 subintervals when moving to a smaller scale.

    At scale $j-1$, we subdivide a given interval $[b,b+m_{j-1})$ into 4 subintervals, separated from each other by a yet to be determined ``buffer'' of $\Delta_{j-1}$ many time steps, taking the form
    \[
    \big[b+k\Delta_{j-1}+(k-1)m_j,b+k\Delta_{j-1}+km_j\big),\quad k\in\{1,2,3,4\},
    \]
    where $m_j=(m_{j-1}-4\Delta_{j-1})/4$. Intuitively, an interval of scale $j-1$ ends up being subdivided into $4$ subintervals of equal length $m_j$, with each subinterval preceded by $\Delta_{j-1}$ many steps; see \Cref{fig:timeintervals}. We will use these steps to allow for the application of  \Cref{thrm:mixing}. An immediate consequence of the above definitions is that $s_j=s_{j-1}(1-4\Delta_{j-1}/m_{j-1})$, meaning that when moving to a smaller scale, we ``lose'' a fraction $4\Delta_{j-1}/m_{j-1}$ of time steps that will no longer be considered in further subdivisions.

    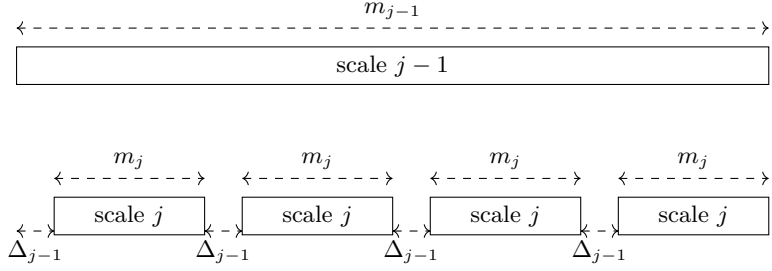
\begin{figure}[!ht]
	\begin{center}
	\begin{tikzpicture}
		\tikzstyle{every node}=[font=\small]
		\draw  (0,0) rectangle  node {\small scale $j-1$} (10,0.5);
		\draw [<->,dashed] (0,0.75) -- (5,0.75) node[above] {$m_{j-1}$} -- (10,0.75);
		\draw [<->,dashed] (0,-1.95) -- (0.25,-1.95) node[below] {$\Delta_{j-1}$} -- (0.5,-1.95);
		\draw  (0.5,-2) rectangle  node {\small scale $j$} (2.5,-1.5);
		\draw [<->,dashed] (0.5,-1.25) -- (1.5,-1.25) node[above] {$m_j$} -- (2.5,-1.25);
		
		\draw [<->,dashed] (2.5,-1.95) -- (2.75,-1.95) node[below] {$\Delta_{j-1}$} -- (3,-1.95);
		\draw  (3,-2) rectangle  node {\small scale $j$} (5,-1.5);
		\draw [<->,dashed] (3,-1.25) -- (4,-1.25) node[above] {$m_j$} -- (5,-1.25);
		
		\draw [<->,dashed] (5,-1.95) -- (5.25,-1.95) node[below] {$\Delta_{j-1}$} -- (5.5,-1.95);
		\draw  (5.5,-2) rectangle  node {\small scale $j$} (7.5,-1.5);
		\draw [<->,dashed] (5.5,-1.25) -- (6.5,-1.25) node[above] {$m_j$} -- (7.5,-1.25);
		
		\draw [<->,dashed] (7.5,-1.95) -- (7.75,-1.95) node[below] {$\Delta_{j-1}$} -- (8,-1.95);
		\draw  (8,-2) rectangle  node {\small scale $j$} (10,-1.5);
		\draw [<->,dashed] (8,-1.25) -- (9,-1.25) node[above] {$m_j$} -- (10,-1.25);
	\end{tikzpicture}
	\caption{The temporal multi-scale recursion.}\label{fig:timeintervals}
	\end{center}
    \end{figure}

    We extend the definition of a good sub-box to time intervals; we say an interval of scale $j$ is \emph{good}, if all of the sub-boxes are good for scale $j$ during the $m_j$ many steps contained in the interval.

    Define the sequence $0=\epsilon_1<\epsilon_2<\cdots<\epsilon_\kappa=\epsilon$ satisfying $\epsilon_j-\epsilon_{j-1}=\epsilon/(\kappa-1)$, and define $D_k$, $k\in\{2,\dots,\kappa-1\}$ to be the event that a fraction of at least $(1-\epsilon_j/2)$ time intervals of scale $j$ are good. For scale $\kappa$, define $D_{\kappa}$ to be the event that for a proportion of at least $(1-\epsilon_{\kappa}/2)$ time intervals of scale $\kappa$, for every mark simultaneously, the locations of the vertices in $Q_{V_{\kappa}}=Q_V$ with a given mark contain as a subset an independent Poisson point process of intensity $(1-\xi_{\kappa})\lambda/M$.

    Recall that we want to obtain that $D_{\kappa}$ occurs for $(1-\epsilon)t$ many time steps, so we require that $(1-\epsilon_\kappa/2)s_\kappa>(1-\epsilon)t$. Calculating, the right-hand side guarantees that when $D_{\kappa}$ holds, at least
    \begin{align*}
        \Big(1-\frac{\epsilon_{\kappa}}{2}\Big)s_\kappa=\Big(1-\frac{\epsilon_\kappa}{2}\Big)s_1\prod_{j=1}^{\kappa-1}\Big(1-\frac{4\Delta_j}{m_j}\Big)\geq \Big(1-\frac{\epsilon}{2}\Big)t\Big(1-\sum_{j=1}^{\kappa-1}\frac{4\Delta_j}{m_j}\Big)
    \end{align*}
    many of the time steps satisfy the requirement of the event $\mathcal{A}_i$. Setting $\Delta_j/m_j=\epsilon/(8\kappa)$, the last expression is larger than $(1-\epsilon)t$ as desired.

    Having set the ratio between the lengths of the time intervals and the small gaps in front of them, we proceed to determine how large $\Delta_j$ needs to be with respect to the volume of the sub-boxes $v_j$, so that particles will have enough time to mix well at that scale. Using \Cref{thrm:mixing} to inform this scaling, we set $v_j$ via the relation
    \begin{equation}\label{eq:Delta_v_relationship}
        \Delta_j=C'v_j^{\alpha/d}
    \end{equation}
    for some sufficiently large constant $C'$. This establishes the relationships
    \begin{equation}\label{eq:v_jm_j}
        \frac{v_j^{\alpha/d}}{m_j}=\frac{\epsilon}{8C'\kappa}
    \end{equation}
    and
\begin{equation}\label{eq:v_geometric}
        v_{j+1}=v_j\Big(\frac{1}{4}-\frac{\epsilon}{8\kappa}\Big)^{d/\alpha}.
    \end{equation}
    As $m_1=t$, this gives $v_1^{\alpha/d}=\frac{\epsilon}{8C'\kappa}t$, which is in agreement with the requirement that $v_1=\omega(\log(t))$. 
    Using the definition of $\kappa$, we also get
    \[
        v_\kappa^{\alpha/d}=\frac{\epsilon}{8C'\kappa}t \Big(\frac{1}{4}-\frac{\epsilon}{8\kappa}\Big)^{(\kappa-1)}.
    \]
    Depending on the choice of $f_\kappa$ in the definition of $\kappa$, this expression can range from $\omega(1)$ (for $f_\kappa(t)$ decreasing to $0$ arbitrarily slowly) to $\Theta(t^{1-c})$,
    where $c$ is any small constant, achieved by setting $f_\kappa(t)$ constant and sufficiently close to $\alpha/\log 2$.
    Consequently, we can ensure that $\lim_{t\to\infty}v_\kappa=\infty$ can be achieved at any desired sublinear rate.

    \paragraph{Step 3: Applying the decoupling theorem.} With the relationship between $\Delta_j$ and $v_j$ established, we proceed to show that if at time $b'=b-\Delta_j$ all subcubes are good for scale $j-1$, and $v_j=\omega(\log t)$, %\todo{condition on $t$ and $v_\kappa$}
     then the interval $[b,b+m_j)$ of scale $j$ is good with probability larger than $1-V_j\exp(-c\lambda v_j+\log(m_j))$ for a positive constant $c$, uniformly across all events that are measurable with respect to the $\sigma$-algebra induced by the particle behaviour up to and including the time $b'$.

    Let 
    \[
        E=\{\text{at time }b'\text{ all subcubes are good for the scale }j-1\}
    \]
    and let $F$ be any event measurable with respect to the above mentioned $\sigma$-algebra. If $E\cap F=\emptyset$, then
    \[
        \P([b,b+m_j)\text{ is not good},E|F)=0,
    \]
    and so the claimed bound is satisfied. Consider therefore the case $E\cap F\neq\emptyset$ and let $\pi_{b'}^u$ be the point process of vertices with mark $u\in\{1,\dots,M\}$, at time $b'$ after conditioning on $E\cap F$. Fix now an arbitrary time $b^*\in[b,b+m_j)$. We proceed to bound the probability
    \[
        \P(\text{at time }b^*\text{ not all subcubes are good for scale }j|E,F).
    \]
    By conditioning on $E$, all subcubes are good for scale $j-1$ at time $b'$. Now, we pick a constant $c_\epsilon$ such that $(1-c_\epsilon)^2(1-\xi_{j-1})=1-\xi_j$, yielding $c_\epsilon=\Theta(\xi_j-\xi_{j-1})$. We also pick a constant $c'$ and set $C'$ from \eqref{eq:Delta_v_relationship} such that
\begin{align}\label{eq:Delta_V_relationship}
        V_j^{1/d}\leq V_{j-1}^{1/d}-c'\Big(\Delta_{j-1}\log\frac{1}{c_\epsilon}\Big)^{1/\alpha},
    \end{align}
    which allows for the application of \Cref{thrm:mixing} with $K=V_{j-1}^{1/d}$ and $K'=V_j^{1/d}$ across each of the $M$ many marks. Each mark $u$ therefore gets a fresh Poisson point process $\Xi^u$ of intensity $(1-c_\epsilon)(1-\xi_{j-1})\lambda/M$, which can be coupled with $\pi_{b^*}^u$ in such a way that $\Xi^u$ is stochastically dominated by $\pi_{b^*}^u$ for all marks $u\in \{1,\dots,M\}$ inside $Q_{V_j}$ with probability at least
    \[
        1-M\exp\big(-c_1c_\epsilon^2\tfrac{\lambda}{M}(1-\xi_{j-1})v_j\big)
    \]
    for some constant $c_1$.

    In order for~\eqref{eq:Delta_V_relationship} to be well defined, we set the constant $B$ from the definition of $V_1$ to be
    \begin{equation}\label{eq:constantB}
        B:=c'(C')^{1/\alpha}(\log(1/c_\epsilon))^{1/\alpha}.
    \end{equation}
    This ensures that~\eqref{eq:Delta_V_relationship} holds with equality for every scale.
    This, the choice $\kappa=O(\log t)$, and~\eqref{eq:Delta_v_relationship} ensure that it is always possible to choose $V_j$ satisfying~\eqref{eq:Delta_V_relationship}, provided $V=\omega(\log^{d/\alpha}(t))$ as assumed in the statement of the theorem.

    Recall that $D_\kappa$ is defined slightly differently compared to the remaining events $D_j$, $j<\kappa$ and that the above bound is in fact already a bound for $D_\kappa$. We therefore focus on the remaining $j<\kappa$. A given sub-box is good for scale $j$ at time $b^*$ if for all marks, $\Xi^u$ contains at least $(1-\xi_j)\lambda v_j/M$ vertices in that sub-box, which, by our choice of $c_\epsilon$ and applying a simple Chernoff and union bound, occurs with probability at least
    \[
    1-M\exp\big(-c_2c_\epsilon^2(1-c_\epsilon)(1-\xi_{j-1})\tfrac{\lambda}{M} v_j\big).
    \]
    Taking a union bound across the time steps in $[b,b+m_j)$ and all sub-boxes of scale $j$, we therefore obtain for some constant $c$ the bound
    \begin{align}\label{eq:interval_good}
        \P([b,b+m_j)\text{ is good},E|F)\geq 1- Mm_j\tfrac{V_j}{v_j}{\rm e}^{-c\lambda v_j/M}\geq 1-V_j{\rm e}^{-c\lambda v_j+\log(m_j)-\log(v_j))}.
    \end{align}

    \paragraph{Step 4: Moving to smaller scale.} Using \eqref{eq:interval_good}, we proceed to bound the probability of the event $D_j^c\cap D_{j-1}$ for $j\geq 2$. When $D_{j-1}$ occurs, there are at least $(1-\epsilon_{j-1}/2)s_{j-1}/m_{j-1}$ intervals that are good for scale $j-1$. By definition, going to scale $j$ means that they will be subdivided into $4(1-\epsilon_{j-1}/2)s_{j-1}/m_{j-1}$ subintervals of scale $j$. In order for $D_j$ to not occur, there should be fewer than $(1-\epsilon_{j}/2)s_j/m_j$ good intervals of scale $j$ among them. Let
    \[
        w=\frac{s_j}{m_j}\Big(\frac{\epsilon_j-\epsilon_{j-1}}{2}\Big)
    \]
    be the difference between these values. If we denote by $Z$ the number of subintervals of $[b,b+m_j)$ of scale $j$ that are not good for scale $j$, but for which the particle configuration at time $b-\Delta_{j-1}$ was good for scale $j-1$, then if $D_{j-1}$ and $D_j^c$ hold simultaneously, then $\{Z\geq w\}$ has to hold.

    By the above, $Z$ can be written as a sum of $\frac{s_j}{m_j}$ indicator variables $\chi_r$, each one representing a scale $j$ time interval. These indicator variables are not independent. However, they are sequential in the sense that the related time intervals follow one after another and do not overlap. Consequently, we can apply the bound from \eqref{eq:interval_good}, which gives the bound $\rho_j$ for the probability of the event $\{\chi_r=1\}$, uniformly across any realisation on the preceding $r-1$ intervals. Using this, $Z$ is stochastically dominated by an independent binomial random variable $Z'$ with parameters $s_j/m_j$ and $\rho_j$. Using a Chernoff bound for a binomial random variable, we obtain
    \begin{align*}
        \P(Z'\geq w)&=\P\big(Z'-\E Z' \geq\tfrac{s_j}{m_j}(\tfrac{\epsilon_j-\epsilon_{j-1}}{2}-\rho_j)\big)\leq\exp\big(-\tfrac{s_j}{m_j}(\tfrac{\epsilon_j-\epsilon_{j-1}}{2})\big(\log(\tfrac{\epsilon_j-\epsilon_{j-1}}{2\rho_j})-1\big)\big).
    \end{align*}
    Recall now that $\epsilon_j-\epsilon_{j-1}=\epsilon/(\kappa-1)$ and $-\log(\rho_j)=\Theta(c \lambda v_j-\log(m_j)-\log(V_j)+\log(v_j))$. Furthermore, recall that $v_j\geq v_\kappa$ and $\kappa=O(\log t)$, which gives
    \[
        \P(Z'\geq w)\leq V_j\exp(-c\tfrac{s_j}{m_j}\tfrac{\epsilon}{\kappa}(c\lambda v_j-\log(m_j)+\log(v_j)).
    \]

    By \eqref{eq:v_jm_j}, we have $v_j^{\alpha/d}/m_j=\epsilon/(8C'\kappa)$ and by definition that $s_j=\Theta(t)$ for all $j$. This yields
    \[
        \P(Z'\geq w)\leq V_j\exp\big(-c\tfrac{\epsilon}{\kappa}\tfrac{\epsilon}{8C'\kappa}t(cv_j^{1-\alpha/d})\big).
    \]
    We conclude that
    \[
        \P(D_j^c\cap D_{j-1})\leq V_j\exp\big(-c t(v_j^{1-\alpha/d})\big).
    \]
    Using $V_j\leq V_1$, $v_j\geq v_\kappa$, and $v_{\kappa}^{\alpha/d}=\omega(f(t))$ for any function $f(t)$ diverting arbitrarily slowly to infinity, we obtain
    \begin{equation}\label{eq:D_j^c}
        \P(D_j^c\cap D_{j-1})\leq V_1\exp\big(-ct/f(t)\big),
    \end{equation}
    where the constant $c$ has been changing throughout the computation.
    We are now ready to conclude the proof, by bounding
    \[
        \P(D_{\kappa}^c)\leq \P(D_\kappa^c\cap D_{\kappa-1})+\P(D_{\kappa-1}^c).
    \]
    Applying this inequality recursively, we get
    \[
        \P(D_{\kappa}^c)\leq \sum_{j=2}^\kappa\P(D_j^c\cap D_{j-1})+\P(D_1^c)
    \]
    and using \eqref{eq:D_1}, \eqref{eq:D_j^c} and recalling the equality $v_1^{\alpha/d}=\frac{\epsilon}{8C'\kappa}t$ and the definition of $\kappa$, we get
    \[
        \P(D_{\kappa}^c)\leq V_1\exp\big(-ct/f(t)\big).
    \]
    Recall that $V_1^{1/d}=V^{1/d}+B\sum_{i=1}^{\kappa-1}v_i^{1/d}$ and that $v_i$ form a geometric sequence via \eqref{eq:v_geometric}, meaning that the sum is $O(v_1^{1/d})$ provided $t$ is large enough for the prefactor in \eqref{eq:v_geometric} to be sufficiently close to $1/4$. As noted immediately following \eqref{eq:Delta_V_relationship}, $v_1^{\alpha/d}=\frac{\epsilon}{8C'\kappa}t$. Combined, we can bound $V_1^{1/d}$ from above by
    \(
        V^{1/d}+cB(t/\kappa)^{1/\alpha}.
    \)
    Using $\kappa=O(\log t)$ concludes the proof.
\end{proof}

\paragraph{Acknowledgement.}
This research was supported by the Leibniz Association within the Leibniz Junior Research Group on {\em Probabilistic Methods for Dynamic Communication Networks} as part of the Leibniz Competition (grant no.\ J105/2020) and by the Deutsche Forschungsgemeinschaft (DFG, German Research Foundation) under Germany's Excellence Strategy -- The Berlin Mathematics Research Center MATH+ (EXC-2046/1, EXC-2046/2, project ID: 390685689) through the project {\em EF45-3} on {\em Data Transmission in Dynamical Random Networks}.

\section*{References}
\phantomsection
\addcontentsline{toc}{section}{References}
\renewcommand*{\bibfont}{\normalsize}
\printbibliography[heading = none]
\end{document}